\documentclass{IEEEtran}
\usepackage{cite}
\usepackage{url}
\usepackage{amsmath,amssymb,amsfonts}
\usepackage{multirow}
\usepackage{amsthm}
\usepackage{algorithmic}
\usepackage{graphicx}
\usepackage{textcomp}
\usepackage{epsfig}
\usepackage{algorithmic}
\usepackage{epstopdf}
\usepackage[ruled]{algorithm}
\usepackage{enumerate}
\usepackage{float}
\usepackage{graphics,graphicx}
\usepackage{subfigure}
\usepackage{array}
\usepackage{changes}

\usepackage[justification=centering]{caption}

\DeclareGraphicsExtensions{.png}
\graphicspath{{./Pics/}}
\usepackage{color}
\usepackage{xspace}

\def\BibTeX{{\rm B\kern-.05em{\sc i\kern-.025em b}\kern-.08em
    T\kern-.1667em\lower.7ex\hbox{E}\kern-.125emX}}

\newcommand{\mb}[1]{\mathbf{{#1}}}

\newtheorem{lemma}{Lemma}
\newtheorem{theorem}{Theorem}

\newtheorem{rem}{Remark}
\newtheorem{assumption}{Assumption}
\newtheorem{proposition}{Proposition}

\newtheorem{example}{Example}

\begin{document}
	
\title{\LARGE\bf A Unifying Approximate Method of Multipliers \\for Distributed Composite Optimization}

\author{Xuyang Wu and Jie Lu\thanks{X. Wu is with the Division of Decision and Control Systems, KTH Royal Institute of Technology, SE-100 44 Stockholm, Sweden. Email: {\tt xuyangw@kth.se}.}
\thanks{J. Lu is with the School of Information Science and Technology, ShanghaiTech University, 201210 Shanghai, China. Email: {\tt lujie@shanghaitech.edu.cn}.}
\thanks{This work has been supported by the National Natural Science Foundation of China under grant 61603254.}}
\maketitle

\begin{abstract}
This paper investigates solving convex composite optimization on an undirected network, where each node, privately endowed with a smooth component function and a nonsmooth one, is required to minimize the sum of all the component functions throughout the network. To address such a problem, a general Approximate Method of Multipliers (AMM) is developed, which attempts to approximate the Method of Multipliers by virtue of a surrogate function with numerous options. We then design the possibly nonseparable, time-varying surrogate function in various ways, leading to different distributed realizations of AMM. We demonstrate that AMM generalizes more than ten state-of-the-art distributed optimization algorithms, and certain specific designs of its surrogate function result in a variety of new algorithms to the literature. Furthermore, we show that AMM is able to achieve an $O(1/k)$ rate of convergence to optimality, and the convergence rate becomes linear when the problem is locally restricted strongly convex and smooth. Such convergence rates provide new or stronger convergence results to many prior methods that can be viewed as specializations of AMM.
\end{abstract}


\section{Introduction}\label{sec:intro}

This paper addresses the following convex composite optimization problem:
\begin{equation}\label{eq:prob}
\begin{split}
\underset{x\in \mathbb{R}^d}{\operatorname{minimize}} ~&~ \sum_{i=1}^N \left(f_i(x)+h_i(x)\right),
\end{split}
\end{equation}
where each $f_i:\mathbb{R}^d\rightarrow \mathbb{R}$ is convex and has a Lipschitz continuous gradient, and each $h_i:\mathbb{R}^d\rightarrow \mathbb{R}\cup\{+\infty\}$ is convex and can be non-differentiable. Notice that $f_i$ is allowed to be a zero function. Also, if $h_i$ contains an indicator function $\mathcal{I}_{X_i}$ with respect to a closed convex set $X_i\subset\mathbb{R}^d$, i.e., $\mathcal{I}_{X_i}(x)=0$ if $x\in X_i$ and $\mathcal{I}_{X_i}(x)=+\infty$ otherwise, then problem~\eqref{eq:prob} turns into a nonsmooth, constrained convex program.

We consider solving problem~\eqref{eq:prob} in a \emph{distributed} way over a network modeled as a connected, undirected graph $\mathcal{G}=(\mathcal{V}, \mathcal{E})$, where $\mathcal{V}=\{1,\ldots,N\}$ is the vertex set and $\mathcal{E}\subseteq \{\{i,j\}|~i,j\in \mathcal{V},~i\ne j\}$ is the edge set. We suppose each node $i\in\mathcal{V}$ possesses two private component functions $f_i$ and $h_i$, aims at solving \eqref{eq:prob}, and only exchanges information with its neighbors denoted by the set $\mathcal{N}_i=\{j|~\{i,j\}\in \mathcal{E}\}$. Applications of such a distributed composite optimization problem include control of multi-agent systems, robust estimation by sensor networks, resource allocation in smart grids, decision making for swarm robotics, as well as distributed learning \cite{Nedic15b}.

To date, a large body of distributed optimization algorithms have been proposed to solve problem \eqref{eq:prob} or its special cases. A typical practice is to reformulate \eqref{eq:prob} into a constrained problem with a separable objective function and a consensus constraint to be relaxed. For example, the inexact methods \cite{Bajovic17,Mansoori20} solve this reformulated problem by penalizing the consensus constraint in the objective function, and the primal-dual methods \cite{Shi14,ShiW15,LiuYH19,LiZ19,Notarnicola17,Makhdoumi17,Zeng17,Aybat18,ShiW15a,Xu18,Lei16,Wu19,Hong17,Bianchi16,Mokhtari16,Mokhtari16a,WuX20,Xu20,alghunaim20} relax the consensus constraint through dual decomposition techniques. Another common approach is to emulate the centralized subgradient/gradient methods in a decentralized manner, such as the distributed subgradient methods \cite{Nedic15,Xi17,Yuan18} and the distributed gradient-tracking-based methods \cite{Nedic17,QuG18,PuS20,Xin20,Varagnolo16,Lorenzo16,Daneshmand19,Scutari19,Notarnicola2020} where \cite{Lorenzo16,Daneshmand19,Scutari19,Notarnicola2020} utilize surrogate functions in order to realize the decentralized emulation.

Despite the growing literature, relatively few methods manage to tackle the \emph{general form} of problem~\eqref{eq:prob} under distributed settings, among which only the first-order methods \cite{Notarnicola17,Wu19,LiuYH19,Zeng17,Bianchi16,Makhdoumi17,Aybat18,ShiW15a,Xu18,Hong17,LiZ19} are guaranteed to converge to optimality with constant step-sizes. Here, \cite{Notarnicola17,Wu19} rely on strong convexity of the problem, and \cite{LiuYH19,Zeng17,Bianchi16} only prove asymptotic convergence for non-strongly convex problems. In contrast, the remaining works \cite{Makhdoumi17,Aybat18,ShiW15a,Xu18,Hong17,LiZ19} establish $O(1/k)$ convergence rates in solving \eqref{eq:prob}. Although these algorithms are developed using different rationales, we discover that the majority of them, i.e., \cite{Makhdoumi17,Aybat18,ShiW15a,Xu18,Hong17}, can indeed be thought to originate from the Method of Multipliers \cite{Boyd11}. 

The Method of Multipliers is a seminal (centralized) optimization method, and one of its notable variants is the Alternating Direction Method of Multipliers (ADMM) \cite{Boyd11}. In this paper, we develop a novel paradigm of solving \eqref{eq:prob} via approximating the behavior of the Method of Multipliers, called \emph{Approximate Method of Multipliers} (AMM). The proposed AMM adopts a possibly time-varying surrogate function to take the place of the smooth objective function at every minimization step in the Method of Multipliers, facilitating abundant designs of distributed algorithms for solving \eqref{eq:prob} in a fully decentralized fashion. Unlike the typical separable surrogate functions for distributed optimization such as those in \cite{Lorenzo16,Daneshmand19,Scutari19,Notarnicola2020}, our surrogate function can be nonseparable in some distributed versions of AMM. It also admits certain function forms that cannot be included by \cite{Lorenzo16,Daneshmand19,Scutari19,Notarnicola2020}.

To enable distributed implementation of AMM, we first opt for a Bregman-divergence-type surrogate function, leading to a class of distributed realizations of AMM referred to as \emph{Distributed Approximate Method of Multipliers} (DAMM). We also utilize convex conjugate and graph topologies to design the surrogate function, and thus construct two additional sets of distributed realizations of AMM, referred to as DAMM-SC and DAMM-SQ, for solving smooth convex optimization (i.e., \eqref{eq:prob} with $h_i\equiv0$ $\forall i\in\mathcal{V}$). We concretely exemplify such realizations of AMM, so that new distributed proximal/second-order/gradient-tracking methods can be obtained. Apart from that, AMM and its distributed realizations also unify a wide range of state-of-the-art distributed first-order and second-order algorithms \cite{ShiW15,Xu18,Mokhtari16a,Hong17,ShiW15a,Aybat18,Shi14,Makhdoumi17,Lei16,Nedic17,Mokhtari16} for solving \eqref{eq:prob} or its special cases, as all of them can be cast into the form of AMM/DAMM/DAMM-SQ. This offers a unifying perspective for understanding the nature of these existing methods with various design rationales.

We show that AMM allows the nodes to reach a consensus and converge to the optimal value at a rate of $O(1/k)$, either with a particular surrogate function type or under a local restricted strong convexity condition on $\sum_{i\in \mathcal{V}} f_i(x)$. For all the algorithms that are able to solve the general form of problem~\eqref{eq:prob} in a distributed way, the convergence rates of AMM and the algorithms in \cite{Makhdoumi17,Aybat18,ShiW15a,Xu18,Hong17,LiZ19} achieve the best order of $O(1/k)$, while the algorithms in \cite{Makhdoumi17,Aybat18,ShiW15a,Xu18,Hong17} are indeed special cases of AMM with that particular surrogate function type and the convergence rate in \cite{LiZ19} is in terms of an implicit measure of optimality error. Moreover, we show that when problem~\eqref{eq:prob} is both smooth and locally restricted strongly convex, AMM enables the nodes to attain the optimum at a linear rate. Unlike most existing works, this linear convergence is established in no need of global strong convexity. Naturally, our analysis on AMM yields new convergence results and relaxes some problem assumptions with no degeneration of the convergence rate order for quite a few existing distributed optimization methods that can be generalized by AMM.

The outline of the paper is as follows: Section~\ref{sec:DAMM} describes the proposed AMM and the distributed designs. Section~\ref{sec:connection} demonstrates that AMM generalizes a number of prior distributed optimization methods. Section~\ref{sec:convergence} discusses the convergence results of AMM under various assumptions, and Section~\ref{sec:numerical} presents the comparative simulation results. Finally, Section~\ref{sec:conclusion} concludes the paper.

\textbf{Notation and Definition}: We use $A=(a_{ij})_{n\times n}$ to denote an $n\times n$ real matrix whose $(i,j)$-entry, denoted by $[A]_{ij}$, is equal to $a_{ij}$. In addition, $\operatorname{Null}(A)$ is the null space of $A$, $\operatorname{Range}(A)$ is the range of $A$, and $\|A\|$ is the spectral norm of $A$. Besides, $\operatorname{diag}(D_1,\ldots,D_n)\in\mathbb{R}^{nd\times nd}$ represents the block diagonal matrix with $D_1,\ldots,D_n\in \mathbb{R}^{d\times d}$ being its diagonal blocks. Also, $\mathbf{0}$, $\mathbf{1}$, and $\mathbf{O}$ represent a zero vector, an all-one vector, and a zero square matrix of proper dimensions, respectively, and $I_n$ represents the $n\times n$ identity matrix. For any two matrices $A$ and $B$, $A\otimes B$ is the Kronecker product of $A$ and $B$. If $A$ and $B$ are square matrices of the same size, $A\succeq B$ means $A-B$ is positive semidefinite and $A\succ B$ means $A-B$ is positive definite. For any $A=A^T\in\mathbb{R}^{n\times n}$, $\lambda_{\max}(A)$ denotes the largest eigenvalue of $A$. For any $z\in \mathbb{R}^n$ and $A\succeq\mathbf{O}$, $\|z\|=\sqrt{z^Tz}$ and $\|z\|_A=\sqrt{z^TAz}$. For any countable set $S$, we use $|S|$ to represent its cardinality. For any convex function $f:\mathbb{R}^d\rightarrow\mathbb{R}$, $\partial f(x)\subset\mathbb{R}^d$ denotes its subdifferential (i.e., the set of subgradients) at $x$. If $f$ is differentiable, $\partial f(x)$ only contains the gradient $\nabla f(x)$.

Given a convex set $X\subseteq \mathbb{R}^d$, a function $f:\mathbb{R}^d \rightarrow \mathbb{R}$ is said to be \emph{strongly convex} on $X$ with \emph{convexity parameter} $\sigma>0$ (or simply $\sigma$-strongly convex on $X$) if $\langle g_x-g_y, x-y\rangle \ge\sigma\|x-y\|^2$ $\forall x,y\in X$ $\forall g_x\in\partial f(x)$ $\forall g_y\in \partial f(y)$. We say $f$ is \emph{(globally) strongly convex} if it is strongly convex on $\mathbb{R}^d$. Given $\tilde{x}\in X$, $f$ is said to be \emph{restricted strongly convex} with respect to $\tilde{x}$ on $X$ with convexity parameter $\sigma>0$ if $\langle g_x-g_{\tilde{x}}, x-\tilde{x}\rangle \ge \sigma\|x-\tilde{x}\|^2$ $\forall x\in X$ $\forall g_x\in\partial f(x)$ $\forall g_{\tilde{x}}\in\partial f(\tilde{x})$. %We say $f$ is \emph{locally restricted strongly convex} with respect to $\tilde{x}$ if it is restricted strongly convex with respect to $\tilde{x}$ on every convex and compact subset of $\mathbb{R}^d$ containing $\tilde{x}$. 
Given $L\ge0$, $f$ is said to be \emph{$L$-smooth} if it is differentiable and its gradient is Lipschitz continuous with Lipschitz constant $L$, i.e., $\|\nabla f(x)-\nabla f(y)\| \le L\|x-y\|$ $\forall x,y\in \mathbb{R}^d$.

\section{Approximate Method of Multipliers and Distributed Designs}\label{sec:DAMM}

This section develops distributed optimization algorithms for solving problem~\eqref{eq:prob} over the undirected, connected graph $\mathcal{G}$, with the following formalized problem assumption.

\begin{assumption}\label{asm:problem}
For each $i\in\mathcal{V}$, $f_i:\mathbb{R}^d\rightarrow \mathbb{R}$ and $h_i:\mathbb{R}^d\rightarrow \mathbb{R}\cup\{+\infty\}$ are convex. In addition, $f_i$ is $M_i$-smooth for some $M_i>0$. Moreover, there exists at least one optimal solution $x^\star$ to problem \eqref{eq:prob}.
\end{assumption}

\subsection{Approximate Method of Multipliers}\label{ssec:AMM}

We first propose a family of optimization algorithms that effectively approximate the Method of Multipliers \cite{Boyd11} and serve as a cornerstone of developing distributed algorithms for solving problem~\eqref{eq:prob}. 

Let each node $i\in \mathcal{V}$ keep a copy $x_i\in\mathbb{R}^d$ of the global decision variable $x\in\mathbb{R}^d$ in problem~\eqref{eq:prob}, and define
\begin{displaymath}
f(\mathbf{x}):=\textstyle{\sum_{i\in \mathcal{V}}} f_i(x_i),\quad h(\mathbf{x}):=\textstyle{\sum_{i\in \mathcal{V}}} h_i(x_i),
\end{displaymath}
where $\mathbf{x} = (x_1^T, \ldots, x_N^T)^T\in\mathbb{R}^{Nd}$. Then, problem~\eqref{eq:prob} can be equivalently transformed into
\begin{equation}\label{eq:equivprob}
\begin{split}
\underset{\mathbf{x}\in \mathbb{R}^{Nd}}{\operatorname{minimize}}  ~~&~  f(\mathbf{x})+h(\mathbf{x})\\
\operatorname{subject~to}  ~&~ \mathbf{x} \in S:=\{\mathbf{x}\in\mathbb{R}^{Nd}|~x_1 = \cdots = x_N\}.
\end{split}
\end{equation}
Next, let $\tilde{H}\in\mathbb{R}^{Nd\times Nd}$ be such that $\tilde{H}=\tilde{H}^T$, $\tilde{H}\succeq \mathbf{O}$, and $\operatorname{Null}(\tilde{H})=S$. It has been shown in \cite{Mokhtari16} that the consensus constraint $\mathbf{x}\in S$ in problem~\eqref{eq:equivprob} is equivalent to $\tilde{H}^{\frac{1}{2}}\mathbf{x}=\mathbf{0}$. Therefore, \eqref{eq:equivprob} is identical to
\begin{equation}\label{eq:eqprob}
\begin{split}
\underset{\mathbf{x}\in \mathbb{R}^{Nd}}{\operatorname{minimize}}  ~~&~  f(\mathbf{x})+h(\mathbf{x})\\
\operatorname{subject~to}  ~&~ \tilde{H}^{\frac{1}{2}}\mathbf{x} = \mathbf{0}.
\end{split}
\end{equation}
Problems~\eqref{eq:prob}, \eqref{eq:equivprob}, and \eqref{eq:eqprob} have the same optimal value. Also, $x^\star\in \mathbb{R}^d$ is an optimum of \eqref{eq:prob} if and only if $\mathbf{x}^\star=((x^\star)^T, \ldots, (x^\star)^T)^T\in \mathbb{R}^{Nd}$ is an optimum of \eqref{eq:equivprob} and \eqref{eq:eqprob}.

Consider applying the Method of Multipliers \cite{Boyd11} to solve problem~\eqref{eq:eqprob}, which gives
\begin{align*}
\mathbf{x}^{k+1} &\in \operatorname{\arg\;\min}_{\mathbf{x}\in\mathbb{R}^{Nd}}f(\mathbf{x})\!+h(\mathbf{x})\!+\frac{\rho}{2}\|\mathbf{x}\|_{\tilde{H}}^2\!+(\mathbf{v}^k)^T\tilde{H}^{\frac{1}{2}}\mathbf{x},\displaybreak[0]\\
\mathbf{v}^{k+1} &= \mathbf{v}^k+\rho\tilde{H}^{\frac{1}{2}}\mathbf{x}^{k+1}.
\end{align*}
Here, $\mathbf{x}^{k}\in\mathbb{R}^{Nd}$ is the primal variable at iteration $k\ge0$, which is updated by minimizing an augmented Lagrangian function with the penalty $\frac{\rho}{2}\|\mathbf{x}\|_{\tilde{H}}^2$, $\rho>0$ on the consensus constraint $\tilde{H}^{\frac{1}{2}}\mathbf{x} = \mathbf{0}$. In addition, $\mathbf{v}^{k}\in\mathbb{R}^{Nd}$ is the dual variable, whose initial value $\mathbf{v}^0$ can be arbitrarily set. 

Although the Method of Multipliers may be applied to solve \eqref{eq:eqprob} with properly selected parameters, it is not implementable in a distributed fashion over $\mathcal{G}$, even when the problem reduces to linear programming. To address this issue, our strategy is to first derive a paradigm of approximating the Method of Multipliers and then design its distributed realizations.

Our approximation approach is as follows: Starting from any $\mathbf{v}^0\in\mathbb{R}^{Nd}$, let
\begin{align}
\mathbf{x}^{k+1} &= \underset{\mathbf{x}\in\mathbb{R}^{Nd}}{\operatorname{\arg\;\min}}\;u^k(\mathbf{x})\!+\!h(\mathbf{x})\!+\!\frac{\rho}{2}\|\mathbf{x}\|_H^2\!+\!(\mathbf{v}^k)^T\tilde{H}^{\frac{1}{2}}\mathbf{x},\label{eq:AMMprimal}\displaybreak[0]\\
\mathbf{v}^{k+1} &= \mathbf{v}^k+\rho\tilde{H}^{\frac{1}{2}}\mathbf{x}^{k+1},\quad\forall k\ge0.\label{eq:AMMdual}
\end{align}
Compared to the Method of Multipliers, we adopt the same dual update \eqref{eq:AMMdual} but construct a different primal update \eqref{eq:AMMprimal}. In \eqref{eq:AMMprimal}, we use a possibly time-varying surrogate function $u^k:\mathbb{R}^{Nd}\rightarrow \mathbb{R}$ to replace $f(\mathbf{x})$ in the primal update of the Method of Multipliers, whose conditions are imposed in Assumption~\ref{asm:ukassumption} below. Additionally, to introduce more flexibility, we use a different weight matrix $H\in\mathbb{R}^{Nd\times Nd}$ to define the penalty term $\frac{\rho}{2}\|\mathbf{x}\|_H^2$, $\rho>0$. We suppose $H$ has the same properties as $\tilde{H}$, i.e., $H=H^T\succeq\mathbf{O}$ and $\operatorname{Null}(H)=S$. 

\begin{assumption}\label{asm:ukassumption}
The functions $u^k$ $\forall k\ge0$ satisfy the following:
\begin{enumerate}[(a)]
\item $u^k\,\forall k\!\ge\!0$ are convex and twice continuously differentiable.
\item $u^k+\frac{\rho}{2}\|\cdot\|_H^2$ $\forall k\ge 0$ are strongly convex, whose convexity parameters are \emph{uniformly} bounded from below by some positive constant. 
\item $\nabla u^k$ $\forall k\ge0$ are Lipschitz continuous, whose Lipschitz constants are \emph{uniformly} bounded from above by some nonnegative constant.
\item $\nabla u^k(\mathbf{x}^k) = \nabla f(\mathbf{x}^k)$ $\forall k\ge0$, where $\mathbf{x}^0$ can be arbitrarily given and $\mathbf{x}^k$, $k\ge1$ is generated by \eqref{eq:AMMprimal}--\eqref{eq:AMMdual}.
\end{enumerate}
\end{assumption}

The strong convexity condition in Assumption~\ref{asm:ukassumption}(b) guarantees that $\mathbf{x}^{k+1}$ in \eqref{eq:AMMprimal} is well-defined and uniquely exists. Assumption~\ref{asm:ukassumption}(d) is the key to making \eqref{eq:AMMprimal}--\eqref{eq:AMMdual} solve problem~\eqref{eq:eqprob}. To explain this, note that \eqref{eq:AMMprimal} is equivalent to finding the unique $\mathbf{x}^{k+1}$ satisfying
\begin{align}
-\nabla u^k(\mathbf{x}^{k+1})-\rho H\mathbf{x}^{k+1}-\tilde{H}^\frac{1}{2}\mathbf{v}^k\in\partial h(\mathbf{x}^{k+1}).\label{eq:AMMprimaloptimalitycondition}
\end{align}
Let $(\mathbf{x}^\star,\mathbf{v}^\star)$ be a primal-dual optimal solution pair of problem~\eqref{eq:eqprob}, which satisfies
\begin{align}\label{eq:zgs_opt}
-\tilde{H}^\frac{1}{2}\mathbf{v}^\star-\nabla f(\mathbf{x}^\star)\in\partial h(\mathbf{x}^\star).
\end{align}
If $(\mathbf{x}^k,\mathbf{v}^k)=(\mathbf{x}^\star,\mathbf{v}^\star)$, then $\mathbf{x}^{k+1}$ has to be $\mathbf{x}^\star$ because of Assumption~\ref{asm:ukassumption}(d), $H\mathbf{x}^\star=\mathbf{0}$, \eqref{eq:zgs_opt}, and the uniqueness of $\mathbf{x}^{k+1}$ in \eqref{eq:AMMprimaloptimalitycondition}. It follows from \eqref{eq:AMMdual} and $\tilde{H}^{\frac{1}{2}}\mathbf{x}^\star=\mathbf{0}$ that $\mathbf{v}^{k+1}=\mathbf{v}^\star$. Therefore, $(\mathbf{x}^\star,\mathbf{v}^\star)$ is a \emph{fixed point} of \eqref{eq:AMMprimal}--\eqref{eq:AMMdual}. The remaining conditions in Assumption~\ref{asm:ukassumption} will be used for convergence analysis later in Section~\ref{sec:convergence}.

The paradigm described by \eqref{eq:AMMprimal}--\eqref{eq:AMMdual} and obeying Assumption~\ref{asm:ukassumption} is called \emph{Approximate Method of Multipliers} (AMM). As there are numerous options of the surrogate function $u^k$, AMM unifies a wealth of optimization algorithms, including a variety of existing methods (cf. Section~\ref{sec:connection}) and many brand new algorithms. Moreover, since Assumption~\ref{asm:ukassumption} allows $u^k$ to have a more favorable structure than $f$, AMM with appropriate $u^k$'s may induce a prominent reduction in computational cost compared to the Method of Multipliers. In the sequel, we will provide various options of $u^k$, which give rise to a series of \emph{distributed} versions of AMM.

A number of existing works \cite{Facchinei15,Facchinei17,Hong17a,Scutari17,Cannelli19,Lorenzo16,Daneshmand19,Scutari19,Notarnicola2020} also introduce surrogate functions to optimization methods, among which \cite{Lorenzo16,Daneshmand19,Scutari19,Notarnicola2020} develop distributed algorithms and can address partially nonconvex problems on time-varying networks. The advantages of AMM as well as its differences from the algorithms in \cite{Lorenzo16,Daneshmand19,Scutari19,Notarnicola2020} are highlighted as follows: 

i) The algorithms carrying surrogate functions in \cite{Lorenzo16,Daneshmand19,Scutari19,Notarnicola2020} are (primal) gradient-tracking-based methods. In contrast, AMM incorporates a surrogate function into a primal-dual framework, so that AMM is inherently different from the algorithms in \cite{Lorenzo16,Daneshmand19,Scutari19,Notarnicola2020} and requires new analysis tools. 

ii) The surrogate function conditions in Assumption~\ref{asm:ukassumption} intersect with but still differ from those in \cite{Lorenzo16,Daneshmand19,Scutari19,Notarnicola2020}. For instance, when $f$ is twice continuously differentiable, $u^k(\mathbf{x})=\langle \nabla f(\mathbf{x}^k),\mathbf{x}\rangle+\frac{1}{2}(\mathbf{x}-\mathbf{x}^k)^T(\nabla^2 f(\mathbf{x}^k)+\epsilon I)(\mathbf{x}-\mathbf{x}^k)$ with $\epsilon>0$ meets Assumption~\ref{asm:ukassumption} but cannot be included by \cite{Lorenzo16,Daneshmand19,Scutari19,Notarnicola2020}.

iii) To enable distributed implementation, the surrogate functions in \cite{Lorenzo16,Daneshmand19,Scutari19,Notarnicola2020} need to be fully separable in the sense that they can be written as the sum of $N$ functions with independent variables. In contrast, AMM allows $u^k$ to be densely coupled under proper design yet can still be executed in a distributed fashion (cf. Sections~\ref{ssec:smooth} and~\ref{sssec:ESOM}). This may lead to more diverse algorithm design.

iv) In \cite{Lorenzo16,Daneshmand19,Scutari19,Notarnicola2020}, it is required that the global nonsmooth objective function be accessible to every node, while problem~\eqref{eq:prob} only assigns a local nonsmooth component $h_i$ to each node $i$, which is more applicable to scenarios concerning privacy and makes the analysis more challenging.

\subsection{Distributed Approximate Method of Multipliers}\label{ssec:DAMM}

This subsection lays out the parameter designs of AMM for distributed implementations.

We first apply the following change of variable to AMM: 
\begin{align}
\mathbf{q}^k=\tilde{H}^{\frac{1}{2}}\mathbf{v}^k, \quad k\ge 0.\label{eq:q=Hv}
\end{align}
Then, AMM \eqref{eq:AMMprimal}--\eqref{eq:AMMdual} can be rewritten as
\begin{align}
\mathbf{x}^{k+1} &=\underset{\mathbf{x}\in\mathbb{R}^{Nd}}{\operatorname{\arg\;\min}}~u^k(\mathbf{x})+h(\mathbf{x})+\frac{\rho}{2}\|\mathbf{x}\|_H^2+(\mathbf{q}^k)^T\mathbf{x},\label{eq:ouralgqformprimal}\displaybreak[0]\\
\mathbf{q}^{k+1} &= \mathbf{q}^k+\rho \tilde{H}\mathbf{x}^{k+1}.\label{eq:ouralgqformdual}
\end{align}
Moreover, note that
\begin{align*}
&\operatorname{Range}(\tilde{H}^{\frac{1}{2}})=\operatorname{Range}(\tilde{H})\displaybreak[0]\\
&=S^\bot:=\{\mathbf{x}\in\mathbb{R}^{Nd}|~x_1+\cdots+x_N=\mathbf{0}\},
\end{align*}
where $S^\bot$ is the orthogonal complement of $S$ in \eqref{eq:equivprob}. Hence, \eqref{eq:q=Hv} requires $\mathbf{q}^k\in S^\bot$ $\forall k\ge0$, which, due to \eqref{eq:ouralgqformdual}, can be ensured simply by the following initialization:
\begin{align}
\mathbf{q}^0\in S^\bot.\label{eq:q0}
\end{align}
Therefore, \eqref{eq:ouralgqformprimal}--\eqref{eq:q0} is an equivalent form of AMM.

Next, partition the primal variable $\mathbf{x}^k$ and the dual variable $\mathbf{q}^k$ in \eqref{eq:ouralgqformprimal}--\eqref{eq:q0} as $\mathbf{x}^k=((x_1^k)^T,\ldots,(x_N^k)^T)^T$ and $\mathbf{q}^k=((q_1^k)^T,\ldots,(q_N^k)^T)^T$. Suppose each node $i\in\mathcal{V}$ maintains $x_i^k\in\mathbb{R}^d$ and $q_i^k\in\mathbb{R}^d$. Clearly, the nodes manage to collectively meet \eqref{eq:q0} by setting, for instance, $q_i^0=\mathbf{0}$ $\forall i\in\mathcal{V}$. Below, we discuss the selections of $u^k$, $H$, and $\tilde{H}$ for the sake of distributed implementations of \eqref{eq:ouralgqformprimal} and \eqref{eq:ouralgqformdual}.

To this end, we let
\begin{align*}
H=P\otimes I_d,\quad \tilde{H}=\tilde{P}\otimes I_d, 
\end{align*}
and impose Assumption~\ref{asm:weightmatrix} on $P, \tilde{P}\in \mathbb{R}^{N\times N}$. 

\begin{assumption}\label{asm:weightmatrix}
The matrices $P=(p_{ij})_{N\times N}$ and $\tilde{P}=(\tilde{p}_{ij})_{N\times N}$ satisfy the following:
\begin{enumerate}[(a)]
%\item $[P]_{ij}=p_{ij}<0$, $[\tilde{P}]_{ij}=\tilde{p}_{ij}<0$, $\forall i\in \mathcal{V}$, $\forall j\in\mathcal{N}_i$.
\item $p_{ij}=p_{ji}$, $\tilde{p}_{ij}=\tilde{p}_{ji}$, $\forall\{i,j\}\in\mathcal{E}$.
%\item $[P]_{ii}=p_{ii}=-\sum_{j\in \mathcal{N}_i}p_{ij}$, $[\tilde{P}]_{ii}=\tilde{p}_{ii}=-\sum_{j\in \mathcal{N}_i}\tilde{p}_{ij}$, $\forall i\in \mathcal{V}$.
\item $p_{ij}=\tilde{p}_{ij}=0$, $\forall i\in\mathcal{V}$, $\forall j\notin\mathcal{N}_i\cup\{i\}$.
\item $\operatorname{Null}(P)=\operatorname{Null}(\tilde{P})=\operatorname{span}(\mathbf{1})$.
\item $P\succeq \mathbf{O}$, $\tilde{P}\succeq \mathbf{O}$.
\end{enumerate}
\end{assumption}
The nodes can jointly determine $P,\tilde{P}$ under Assumption~\ref{asm:weightmatrix} without any centralized coordination. For instance, we can let each node $i\in\mathcal{V}$ agree with every neighbor $j\in\mathcal{N}_i$ on $p_{ij}=p_{ji}<0$ and $\tilde{p}_{ij}=\tilde{p}_{ji}<0$, and set $p_{ii}=-\sum_{j\in \mathcal{N}_i}p_{ij}$ and $\tilde{p}_{ii}=-\sum_{j\in \mathcal{N}_i}\tilde{p}_{ij}$, which directly guarantee Assumption~\ref{asm:weightmatrix}(a). Then, Assumption~\ref{asm:weightmatrix}(c)(d) are satisfied effortlessly by means of Assumption~\ref{asm:weightmatrix}(b) and the connectivity of $\mathcal{G}$. Typical examples of such $P$ and $\tilde{P}$ include the graph Laplacian matrix $L_{\mathcal{G}}\in\mathbb{R}^{N\times N}$ with $[L_{\mathcal{G}}]_{ij}=[L_{\mathcal{G}}]_{ji}=-1$ $\forall \{i,j\}\in \mathcal{E}$ and the Metropolis weight matrix $M_{\mathcal{G}}\in\mathbb{R}^{N\times N}$ with $[M_{\mathcal{G}}]_{ij}=[M_{\mathcal{G}}]_{ji}=-\frac{1}{\max\{|\mathcal{N}_i|, |\mathcal{N}_j|\}+1}$ $\forall \{i,j\}\in \mathcal{E}$. Also, the conditions on $H$ and $\tilde{H}$ in Section~\ref{ssec:AMM} hold due to Assumption~\ref{asm:weightmatrix}.

With Assumption~\ref{asm:weightmatrix}, \eqref{eq:ouralgqformdual} can be accomplished by letting each node update as
\begin{align}
q_i^{k+1}=q_i^k+\rho\textstyle{\sum_{j\in \mathcal{N}_i\cup\{i\}}}\tilde{p}_{ij}x_j^{k+1},\quad\forall i\in\mathcal{V}.\label{eq:DAMMdualnodei}
\end{align}
This is a local operation, as each node $i\in\mathcal{V}$ only needs to acquire the primal variables associated with its neighbors.

It remains to design the surrogate function $u^k$ so that \eqref{eq:ouralgqformprimal} can be realized in a distributed way. To this end, let $\psi_i^k:\mathbb{R}^d\rightarrow \mathbb{R}$ $\forall i\in \mathcal{V}$ $\forall k\ge 0$ be a set of functions under Assumption~\ref{asm:psi}.

\begin{assumption}\label{asm:psi}
The functions $\psi_i^k$ $\forall i\in \mathcal{V}$ $\forall k\ge 0$ satisfy:
\begin{enumerate}[(a)]
\item $\psi_i^k$ $\forall i\in \mathcal{V}$ $\forall k\ge 0$ are twice continuously differentiable.
\item $\psi_i^k$ $\forall i\in \mathcal{V}$ $\forall k\ge 0$ are strongly convex, whose convexity parameters are \emph{uniformly} bounded from below by some positive constant.
\item $\nabla\psi_i^k$ $\forall i\in \mathcal{V}$ $\forall k\ge 0$ are Lipschitz continuous, whose Lipschitz constants are \emph{uniformly} bounded from above by some positive constant.
\item $(\sum_{i\in \mathcal{V}} \psi_i^k(x_i))-\frac{\rho}{2}\|\mathbf{x}\|_H^2$ $\forall k\ge0$ are convex.
\end{enumerate}
\end{assumption}

To meet Assumption~\ref{asm:psi}(d), it suffices to let $\psi_i^k$ $\forall i\in\mathcal{V}$ $\forall k\ge0$ be any strongly convex functions whose convexity parameters are larger than or equal to $\rho\lambda_{\max}(H)>0$, and one readily available upper bound on $\lambda_{\max}(H)$ is $\max_{i\in\mathcal{V},\;j\in\mathcal{N}_i}|p_{ij}|N$. 

Now define
\begin{align}
\phi^k(\mathbf{x}):=\bigl(\textstyle{\sum_{i\in \mathcal{V}}}\psi_i^k(x_i)\bigr)-\frac{\rho}{2}\|\mathbf{x}\|_H^2,\quad\forall k\ge 0,\label{eq:phi=sumpsi}
\end{align}
and let $D_{\phi^k}(\mathbf{x},\mathbf{x}^k)$ be the Bregman divergence associated with $\phi^k$ at $\mathbf{x}$ and $\mathbf{x}^k$, i.e.,
\begin{align}
D_{\phi^k}(\mathbf{x}, \mathbf{x}^k)=\phi^k(\mathbf{x})-\phi^k(\mathbf{x}^k)-\langle \nabla \phi^k(\mathbf{x}^k), \mathbf{x}-\mathbf{x}^k\rangle.\label{eq:Bregmandiv}
\end{align}
Then, we set $u^k$ as
\begin{align}
u^k(\mathbf{x})=D_{\phi^k}(\mathbf{x}, \mathbf{x}^k)+\langle \nabla f(\mathbf{x}^k), \mathbf{x}\rangle.\label{eq:compositeformofuk}
\end{align}
With Assumption~\ref{asm:psi}, \eqref{eq:compositeformofuk} is sufficient to ensure Assumption~\ref{asm:ukassumption}. To see this, note from Assumption~\ref{asm:psi}(a)(d) that $u^k(\mathbf{x})$ in \eqref{eq:compositeformofuk} is twice continuously differentiable and convex, i.e., Assumption~\ref{asm:ukassumption}(a) holds. Also note from \eqref{eq:Bregmandiv} and \eqref{eq:compositeformofuk} that
\begin{align}
\nabla u^k(\mathbf{x}) = \nabla \phi^k(\mathbf{x})-\nabla \phi^k(\mathbf{x}^k)+\nabla f(\mathbf{x}^k).\label{eq:ukgradient}
\end{align}
This, along with Assumption~\ref{asm:psi} (b)(c), guarantees Assumption~\ref{asm:ukassumption}(b)(c)(d).

To see how \eqref{eq:compositeformofuk} results in a distributed implementation of \eqref{eq:ouralgqformprimal}, note from \eqref{eq:ukgradient} that \eqref{eq:ouralgqformprimal} is equivalent to
\begin{equation*}%\label{eq:primalupdateoptcond}
\mathbf{0}=\nabla \phi^k(\mathbf{x}^{k+1})-\nabla \phi^k(\mathbf{x}^k)+\nabla f(\mathbf{x}^k)+g^{k+1}+\rho H\mathbf{x}^{k+1}+\mathbf{q}^k
\end{equation*}
for some $g^{k+1}\in \partial h(\mathbf{x}^{k+1})$. Then, using \eqref{eq:phi=sumpsi} and the structure of $H$ given in Assumption~\ref{asm:weightmatrix}, this can be written as
\begin{align*}
\mathbf{0}=&\nabla\psi_i^k(x_i^{k+1})+g_i^{k+1}+q_i^k-\nabla\psi_i^k(x_i^k)+\nabla f_i(x_i^k)\displaybreak[0]\\
&+\rho\textstyle{\sum_{j\in \mathcal{N}_i\cup\{i\}}}p_{ij}x_j^k,\quad\forall i\in\mathcal{V},
\end{align*}
where $g_i^{k+1}\in\partial h_i(x_i^{k+1})$. In other words, \eqref{eq:ouralgqformprimal} can be achieved by letting each node $i\in\mathcal{V}$ solve the following strongly convex optimization problem:
\begin{align}
&x_i^{k+1}=\operatorname{\arg\;\min}_{x\in\mathbb{R}^d}~\psi_i^k(x)+h_i(x)\nonumber\displaybreak[0]\\
&\!\!\!\!+\langle x, q_i^k-\nabla \psi_i^k(x_i^k)+\nabla f_i(x_i^k)+\rho\textstyle{\sum_{j\in \mathcal{N}_i\cup\{i\}}}p_{ij}x_j^k\rangle,\label{eq:DAMMprimalnodei}
\end{align}
which can be locally carried out by node $i$. 

Algorithms described by \eqref{eq:q0}, \eqref{eq:DAMMprimalnodei}, and \eqref{eq:DAMMdualnodei} under Assumptions~\ref{asm:weightmatrix} and~\ref{asm:psi} constitute a set of distributed realizations of AMM, referred to as \emph{Distributed Approximate Method of Multipliers} (DAMM), which can be implemented by the nodes in $\mathcal{G}$ via exchanging information with their neighbors only. Algorithm~\ref{alg:mainalg} describes how each node acts in DAMM. 

\begin{algorithm}\caption{DAMM}\label{alg:mainalg}
\begin{algorithmic}[1]
\STATE \textbf{Initialization:} 
%\STATE All the nodes agree on $\rho>0$.
%\STATE Each pair of neighbors $\{i,j\}\in\mathcal{E}$ agree on $p_{ij}=p_{ji}$ and $\tilde{p}_{ij}=\tilde{p}_{ji}$. 
\STATE Each node $i\in\mathcal{V}$ selects $q_i^0\in\mathbb{R}^d$ such that $\sum_{i\in \mathcal{V}} q_i^0=\mathbf{0}$ (or simply sets $q_i^0=\mathbf{0}$). 
\STATE Each node $i\in \mathcal{V}$ arbitrarily sets $x_i^0\in \mathbb{R}^d$ and sends $x_i^0$ to every neighbor $j\in \mathcal{N}_i$.
\FOR {$k\ge 0$}
%\STATE Each node $i\in \mathcal{V}$ selects $\psi_i^k:\mathbb{R}^d\rightarrow\mathbb{R}$ such that Assumption~\ref{asm:psi} holds.
\STATE Each node $i\in \mathcal{V}$ computes
$x_i^{k+1}=\operatorname{\arg\;\min}_{x\in\mathbb{R}^d}\psi_i^k(x)+h_i(x)+\langle x, q_i^k-\nabla \psi_i^k(x_i^k)+\nabla f_i(x_i^k)+\rho\sum_{j\in \mathcal{N}_i\cup\{i\}}p_{ij}x_j^k\rangle$.
\STATE Each node $i\in \mathcal{V}$ sends $x_i^{k+1}$ to every neighbor $j\in \mathcal{N}_i$.
\STATE Each node $i\in \mathcal{V}$ computes
$q_i^{k+1}=q_i^k+\rho \sum_{j\in \mathcal{N}_i\cup\{i\}} \tilde{p}_{ij}x_j^{k+1}$.
\ENDFOR
\end{algorithmic}
\end{algorithm}

Finally, we provide two examples of DAMM with two particular choices of $\psi_i^k$.

\begin{example}
\rm For each $i\in\mathcal{V}$, let $\psi_i^k(x)=r_i(x-x_i^k)+\frac{\epsilon_i}{2}\|x\|^2$ $\forall k\ge 0$, where $r_i:\mathbb{R}^d\rightarrow \mathbb{R}$ $\forall i\in \mathcal{V}$ can be any convex, smooth, and twice continuously differentiable functions and $\epsilon_i>0$ is such that $\epsilon_i\ge\rho\lambda_{\max}(H)$. Then, DAMM reduces to a new distributed \emph{proximal} algorithm, with the following local update of $x_i^k$: 
\begin{align*}
x_i^{k+1}=&\operatorname{\arg\;\min}_{x\in\mathbb{R}^d}~r_i(x-x_i^k)+\frac{\epsilon_i}{2}\|x-x_i^k\|^2+h_i(x)\displaybreak[0]\\
&+\langle x, q_i^k-\nabla r_i(\mathbf{0})+\nabla f_i(x_i^k)+\rho\textstyle{\sum_{j\in \mathcal{N}_i\cup\{i\}}}p_{ij}x_j^k\rangle.
\end{align*}
\end{example}

\begin{example}
\rm Suppose $f_i$ $\forall i\in \mathcal{V}$ are twice continuously differentiable. Since $\nabla^2 f_i(x)\preceq M_i I_d$ $\forall x\in\mathbb{R}^d$, we can let each $\psi_i^k(x)=\frac{1}{2}x^T(\nabla^2 f_i(x_i^k)+\epsilon_i I_d)x$, where $\epsilon_i>0$ is such that $\epsilon_i\ge\rho\lambda_{\max}(H)$. Then, the resulting DAMM is a new distributed \emph{second-order} method, with the following local update of $x_i^k$:
\begin{align*}
x_i^{k+1}=&\operatorname{\arg\;\min}_{x\in\mathbb{R}^d}~\frac{1}{2}\|x-x_i^k\|_{\nabla^2 f_i(x_i^k)+\epsilon_iI_d}^2+h_i(x)\displaybreak[0]\\
&+\langle x, q_i^k+\nabla f_i(x_i^k)+\rho\textstyle{\sum_{j\in \mathcal{N}_i\cup\{i\}}}p_{ij}x_j^k\rangle.
\end{align*}
\end{example}

\subsection{Special Case: Smooth Problem}\label{ssec:smooth}

In this subsection, we focus on the smooth convex optimization problem $\min_{x\in\mathbb{R}^d}\sum_{i\in\mathcal{V}}f_i(x)$, i.e., \eqref{eq:prob} with $h_i(x)\equiv 0$ $\forall i\in \mathcal{V}$. For this special case of \eqref{eq:prob}, we provide additional designs of the surrogate function $u^k$ in AMM, leading to a couple of variations of DAMM.

Here, we let $u^k$ still be in the form of \eqref{eq:compositeformofuk}, but no longer require $\phi^k+\frac{\rho}{2}\|\cdot\|_H^2$ be a separable function as in \eqref{eq:phi=sumpsi}. Instead, we construct $\phi^k$ based upon another function $\gamma^k:\mathbb{R}^{Nd}\rightarrow \mathbb{R}$ under Assumption~\ref{asm:gammakassumption}.

\begin{assumption}\label{asm:gammakassumption}
The functions $\gamma^k$ $\forall k\ge 0$ satisfy the following:
\begin{enumerate}[(a)]
\item $\gamma^k$ $\forall k\ge 0$ are twice continuously differentiable. 
\item $\gamma^k$ $\forall k\ge0$ are strongly convex, whose convexity parameters are \emph{uniformly} \linebreak[4] bounded from below by $\underline{\gamma}>0$.
\item $\nabla\gamma^k$ $\forall k\ge0$ are Lipschitz continuous, whose Lipschitz constants are \emph{uniformly} bounded from above by $\bar{\gamma}>0$.
\item $(\gamma^k)^\star(\mathbf{x})-\frac{\rho}{2}\|\mathbf{x}\|_H^2$ $\forall k\ge0$ are convex, where $(\gamma^k)^\star(\mathbf{x})=\sup_{\mathbf{y}\in\mathbb{R}^{Nd}}\langle \mathbf{x},\mathbf{y}\rangle-\gamma^k(\mathbf{y})$ is the convex conjugate function of $\gamma^k$.
\item For any $k\ge0$ and any $\mathbf{x}\in\mathbb{R}^{Nd}$, the $i$th $d$-dimensional block of $\nabla\gamma^k(\mathbf{x})$, denoted by $\nabla_i \gamma^k(\mathbf{x})$, is independent of $x_j$ $\forall j\notin \mathcal{N}_i\cup \{i\}$.
\end{enumerate}
\end{assumption}

From Assumption~\ref{asm:gammakassumption}(b)(c), each $(\gamma^k)^\star$ is $(1/\bar{\gamma})$-strongly convex and $(1/\underline{\gamma})$-smooth \cite{Kakade09}, so that Assumption~\ref{asm:gammakassumption}(d) holds as long as $I_{Nd}/\bar{\gamma}-\rho H\succeq \mathbf{O}$. Now we set
\begin{align}
\phi^k(\mathbf{x})=(\gamma^k)^\star(\mathbf{x})-\frac{\rho}{2}\|\mathbf{x}\|_H^2,\quad\forall k\ge0.\label{eq:phismooth}
\end{align}
Unlike DAMM, $u^k$ given by \eqref{eq:compositeformofuk} and \eqref{eq:phismooth} under Assumption~\ref{asm:gammakassumption} does not necessarily guarantee that $u^k(\cdot)+\frac{\rho}{2}\|\cdot\|_H^2$ is separable. Below, we show that such $u^k$ $\forall k\ge0$ also satisfy Assumption~\ref{asm:ukassumption}, leading to another subclass of AMM. 

To do so, first notice that the strong convexity and smoothness of $(\gamma^k)^\star$ guarantee Assumption~\ref{asm:ukassumption}(b)(c). Also note from \eqref{eq:ukgradient} that Assumption~\ref{asm:ukassumption}(d) is assured. In addition, due to Assumption~\ref{asm:gammakassumption}(d), $\phi^k$ is convex and, thus, so is $u^k$. To show the twice continuous differentiability of $u^k$ in Assumption~\ref{asm:ukassumption}(a), consider the fact from \cite{Beck03} that due to Assumption~\ref{asm:gammakassumption}(b)(c), $\nabla\gamma^k$ is invertible and its inverse function is $(\nabla \gamma^k)^{-1}=\nabla(\gamma^k)^\star$. This, along with \eqref{eq:phismooth}, implies that
\begin{align}
\nabla\phi^k(\mathbf{x})=(\nabla \gamma^k)^{-1}(\mathbf{x})-\rho H\mathbf{x}.\label{eq:gradientphi}
\end{align}
From the inverse function theorem \cite{Spivak65}, $(\nabla \gamma^k)^{-1}$ is continuously differentiable, so that $\phi^k$ and therefore $u^k$ given by \eqref{eq:compositeformofuk} and \eqref{eq:phismooth} are twice continuously differentiable. We then conclude that Assumption~\ref{asm:ukassumption} holds.

Equipped with \eqref{eq:compositeformofuk}, \eqref{eq:phismooth}, and Assumption~\ref{asm:gammakassumption}, we start to derive additional distributed realizations of AMM \eqref{eq:ouralgqformprimal}--\eqref{eq:q0} for minimizing $\sum_{i\in\mathcal{V}}f_i(x)$. As $h_i(x)\equiv 0$ $\forall i\in\mathcal{V}$, \eqref{eq:ouralgqformprimal} is equivalent to
\begin{align}
\nabla \phi^k(\mathbf{x}^{k+1})\!-\!\nabla \phi^k(\mathbf{x}^k)\!+\!\nabla f(\mathbf{x}^k)\!+\!\rho H\mathbf{x}^{k+1}\!+\!\mathbf{q}^k \!= \!\mathbf{0}.\label{eq:phikoptcondsmooth}
\end{align}
This, together with \eqref{eq:gradientphi}, gives
\begin{align}
\mathbf{x}^{k+1}=\nabla \gamma^k(\nabla \phi^k(\mathbf{x}^k)-\nabla f(\mathbf{x}^k)-\mathbf{q}^k).\label{eq:distributedimplementationGk}
\end{align}
Like DAMM, we let each node $i\in\mathcal{V}$ maintain the $i$th $d$-dimensional blocks of $\mathbf{x}^k$ and $\mathbf{q}^k$, i.e., $x_i^k$ and $q_i^k$, where the update of $q_i^k$ is the same as \eqref{eq:DAMMdualnodei}. According to \eqref{eq:distributedimplementationGk}, the update of $x_i^k$ is given by $x_i^{k+1}=\nabla_i\gamma^k(\nabla\phi^k(\mathbf{x}^k)-\nabla f(\mathbf{x}^k)-\mathbf{q}^k)$. From Assumption~\ref{asm:gammakassumption}(e), this can be computed by node $i$ provided that it has access to $\nabla_j\phi^k(\mathbf{x}^k)-\nabla f_j(x_j^k)-q_j^k$ $\forall j\in\mathcal{N}_i\cup\{i\}$, where $\nabla_j\phi^k(\mathbf{x}^k)\in\mathbb{R}^d$ denotes the $j$th block of $\nabla\phi^k(\mathbf{x}^k)$. Therefore, the remaining question is whether each node $i\in\mathcal{V}$ is able to locally evaluate $\nabla_i\phi^k(\mathbf{x}^k)$. In fact, this can be enabled by the following two concrete ways of constructing $\gamma^k$.

\textbf{Way \#1}: Let $\gamma^k=\gamma$ $\forall k\ge 0$ for some $\gamma:\mathbb{R}^{Nd}\rightarrow\mathbb{R}$ satisfying Assumption~\ref{asm:gammakassumption}. Then, according to \eqref{eq:phismooth}, $\phi^k$ $\forall k\ge0$ are fixed to some $\phi:\mathbb{R}^{Nd}\rightarrow\mathbb{R}$. We introduce an auxiliary variable $\mathbf{y}^k=((y_1^k)^T,\ldots,(y_N^k)^T)^T$ such that $\mathbf{y}^k=\nabla \phi(\mathbf{x}^k)$ $\forall k\ge0$. From \eqref{eq:phikoptcondsmooth}, $\mathbf{y}^k$ satisfies the recursive relation
\begin{align*}
\mathbf{y}^{k+1}=\mathbf{y}^k-\nabla f(\mathbf{x}^k)-\mathbf{q}^k-\rho H\mathbf{x}^{k+1}.
\end{align*}
Due to the structure of $H$ in Assumption~\ref{asm:weightmatrix}, each node $i\in\mathcal{V}$ is able to locally update $y_i^{k+1}$ as above. Nevertheless, the initialization $\mathbf{y}^0=\nabla\phi(\mathbf{x}^0)=(\nabla\gamma)^{-1}(\mathbf{x}^0)-\rho H\mathbf{x}^0$ cannot be achieved without centralized coordination, given that $\mathbf{x}^0$ is arbitrarily chosen. We may overcome this by imposing a restriction on $\mathbf{x}^0$ as follows: Let each node $i\in\mathcal{V}$ pick any $\tilde{z}_i\in \mathbb{R}^d$, and force $x_i^0=\nabla_i\gamma(\tilde{\mathbf{z}})$ with $\tilde{\mathbf{z}}=(\tilde{z}_1^T,\ldots,\tilde{z}_N^T)^T$, so that $\tilde{\mathbf{z}}=(\nabla\gamma)^{-1}(\mathbf{x}^0)$. Then, by letting each $y_i^0=\tilde{z}_i-\rho\sum_{j\in \mathcal{N}_i\cup\{i\}} p_{ij}x_j^0$, we obtain $\mathbf{y}^0=\nabla\phi(\mathbf{x}^0)$. Due to Assumption~\ref{asm:gammakassumption}(e), such initialization is decentralized, which only requires each node $i$ to share $\tilde{z}_i$ and $x_i^0$ with its neighbors. 

Incorporating the above into \eqref{eq:distributedimplementationGk} yields another distributed version of AMM, which is called \emph{Distributed Approximate Method of Multipliers for Smooth problems with Constant update functions} (DAMM-SC) and is specified in Algorithm~\ref{alg:DAMMSC}.

\begin{algorithm}\caption{DAMM-SC}\label{alg:DAMMSC}
\begin{algorithmic}[1]
	\STATE \textbf{Initialization:} 
	\STATE Each node $i\in\mathcal{V}$ selects $q_i^0\in\mathbb{R}^d$ such that $\sum_{i\in \mathcal{V}} q_i^0=\mathbf{0}$ (or simply sets $q_i^0=\mathbf{0}$). 
	%\STATE All the nodes agree on $\gamma:\mathbb{R}^{Nd}\rightarrow \mathbb{R}$ satisfying the conditions in Assumption~\ref{asm:gammakassumption}.
	\STATE Each node $i\!\in\!\mathcal{V}$ arbitrarily chooses $\tilde{z}_i\!\in\!\mathbb{R}^d$ and sends $\tilde{z}_i$ to every neighbor $j\!\in\! \mathcal{N}_i$.
	\STATE Each node $i\in\mathcal{V}$ sets $x_i^0\!=\!\nabla_i \gamma(\tilde{\mathbf{z}})$ (which only depends on $\tilde{z}_j$ $\forall j\in\mathcal{N}_i\cup\{i\}$) and sends it to every neighbor $j\in \mathcal{N}_i$.
	\STATE Each node $i\in\mathcal{V}$ sets $y_i^0=\tilde{z}_i-\rho\sum_{j\in \mathcal{N}_i\cup\{i\}} p_{ij}x_j^0$.
	\STATE Each node $i\in\mathcal{V}$ sends $y_i^0-\nabla f_i(x_i^0)-q_i^0$ to every neighbor $j\in \mathcal{N}_i$.
	\FOR {$k\ge 0$}
	\STATE Each node $i\in\mathcal{V}$ computes
	$x_i^{k+1}=\nabla_i \gamma(\mathbf{y}^k-\nabla f(\mathbf{x}^k)-\mathbf{q}^k)$ (which only depends on $y_j^k-\nabla f_j(x_j^k)-q_j^k$ $\forall j\in\mathcal{N}_i\cup\{i\}$).
	\STATE Each node $i\in \mathcal{V}$ sends $x_i^{k+1}$ to every neighbor $j\in \mathcal{N}_i$.
	\STATE Each node $i\in \mathcal{V}$ computes $y_i^{k+1}=y_i^k-\nabla f_i(x_i^k)-q_i^k-\rho \sum_{j\in \mathcal{N}_i\cup\{i\}}p_{ij}x_j^{k+1}$.
	\STATE Each node $i\in \mathcal{V}$ computes $q_i^{k+1}=q_i^k+\rho \sum_{j\in \mathcal{N}_i\cup\{i\}} \tilde{p}_{ij}x_j^{k+1}$.
	\STATE Each node $i\in \mathcal{V}$ sends $y_i^{k+1}-\nabla f_i(x_i^{k+1})-q_i^{k+1}$ to every neighbor $j\in \mathcal{N}_i$.
	\ENDFOR
\end{algorithmic}
\end{algorithm}

\begin{example}
\rm We can further reduce the communication cost of DAMM-SC by restricting $\nabla_i\gamma$ to solely depend on node $i$. For example, let $g_i:\mathbb{R}^d\rightarrow\mathbb{R}$ $\forall i\in\mathcal{V}$ be arbitrary twice continuously differentiable, smooth, convex functions and let $g(\mathbf{x})=\sum_{i\in\mathcal{V}}g_i(x_i)$. Then, we may choose $\gamma(\mathbf{x})=g(\mathbf{x})+\frac{\epsilon}{2}\|\mathbf{x}\|^2$ with $\epsilon>0$. It is straightforward to see that Assumption~\ref{asm:gammakassumption}(a)(b)(c) hold and Assumption~\ref{asm:gammakassumption}(d) can be satisfied with proper $\rho,\epsilon>0$. For such a choice of $\gamma$, $\nabla_i\gamma(\mathbf{x})=\nabla g_i(x_i)+\epsilon x_i$, which is up to node $i$ itself and hence meets Assumption~\ref{asm:gammakassumption}(e). Thus, the primal update of DAMM-SC, i.e., Line~8 of Algorithm~\ref{alg:DAMMSC}, is simplified to $x_i^{k+1}=\nabla g_i(y_i^k-\nabla f_i(x_i^k)-q_i^k)+\epsilon(y_i^k-\nabla f_i(x_i^k)-q_i^k)$, so that each node $i$ only needs to share $x_i^{k+1}$ with its neighbors during each iteration and the local communications in Line~6 and Line~12 of Algorithm~\ref{alg:DAMMSC} are eliminated.
\end{example}

\textbf{Way \#2}: For each $k\ge0$, let $\gamma^k(\mathbf{x})=\frac{1}{2}\mathbf{x}^TG^k\mathbf{x}$, where $G^k=(G^k)^T\in \mathbb{R}^{Nd\times Nd}$. Suppose there exist $\bar{\gamma}\ge\underline{\gamma}>0$ such that $\underline{\gamma}I_{Nd}\preceq G^k\preceq \bar{\gamma}I_{Nd}$ $\forall k\ge0$, and also suppose $(G^k)^{-1}\succeq \rho H$ $\forall k\ge0$. Moreover, we let each $G^k$ have a neighbor-sparse structure, i.e., the $d\times d$ $(i,j)$-block of $G^k$, denoted by $G_{ij}^k$, is equal to $\mathbf{O}\in \mathbb{R}^{d\times d}$ for all $j\notin \mathcal{N}_i\cup\{i\}$. It can be shown that such quadratic $\gamma^k$'s satisfy Assumption~\ref{asm:gammakassumption}. Substituting $\gamma^k(\mathbf{x})=\frac{1}{2}\mathbf{x}^TG^k\mathbf{x}$ into \eqref{eq:distributedimplementationGk} gives
\begin{equation}\label{eq:DAMMSQ}
\mathbf{x}^{k+1} = \mathbf{x}^k-G^k(\rho H\mathbf{x}^k+\nabla f(\mathbf{x}^k)+\mathbf{q}^k),
\end{equation}
which can be executed in a distributed manner due to the neighbor-sparse structure of $G^k$. This leads to an additional distributed version of AMM, which is referred to as \emph{Distributed Approximate Method of Multipliers for Smooth problems with Quadratic update functions} (DAMM-SQ) and is elaborated in Algorithm~\ref{alg:DAMMSQ} where we introduce, for convenience, an auxiliary variable $\mathbf{z}^k=((z_1^k)^T,\ldots,(z_N^k)^T)^T=\rho H\mathbf{x}^k+\nabla f(\mathbf{x}^k)+\mathbf{q}^k$. 

\begin{algorithm}\caption{DAMM-SQ}\label{alg:DAMMSQ}
\begin{algorithmic}[1]
\STATE \textbf{Initialization:} 
\STATE Same as Lines 2--3 in Algorithm~\ref{alg:mainalg}.
\STATE Each node $i\in \mathcal{V}$ computes $z_i^0=\nabla f_i(x_i^0)+q_i^0+\rho \sum_{j\in \mathcal{N}_i\cup \{i\}} p_{ij}x_j^0$ and sends it to every neighbor $j\in \mathcal{N}_i$.
\FOR {$k\ge 0$}
%\STATE All the nodes cooperate to select a legitimate $G^k$.
\STATE Each node $i\in \mathcal{V}$ computes $x_i^{k+1}=x_i^k-\sum_{j\in \mathcal{N}_i\cup\{i\}}G_{ij}^kz_j^k$.
\STATE Each node $i\in \mathcal{V}$ sends $x_i^{k+1}$ to every neighbor $j\in \mathcal{N}_i$.
\STATE Each node $i\in \mathcal{V}$ computes $q_i^{k+1}=q_i^k+\rho \sum_{j\in \mathcal{N}_i\cup\{i\}} \tilde{p}_{ij}x_j^{k+1}$.
\STATE Each node $i\in \mathcal{V}$ computes $z_i^{k+1}=\nabla f_i(x_i^{k+1})+q_i^{k+1}+\rho \sum_{j\in \mathcal{N}_i\cup \{i\}} p_{ij}x_j^{k+1}$.
\STATE Each node $i\in \mathcal{V}$ sends $z_i^{k+1}$ to every neighbor $j\in \mathcal{N}_i$.
\ENDFOR
\end{algorithmic}
\end{algorithm}

\begin{example}
\rm We present a new distributed \emph{gradient-tracking} algorithm as an example of DAMM-SQ, where $G^k=\frac{1}{\rho}(I_N-P)\otimes I_d$ $\forall k\ge 0$ with $P=\tilde{P}\prec I_N$ under Assumption~\ref{asm:weightmatrix}. Note that $P=\tilde{P}\prec I_N$ can be satisfied by letting $I_N-P=I_N-\tilde{P}$ be strictly diagonally dominant with positive diagonal entries. Similar to the discussions below Assumption~\ref{asm:weightmatrix}, such $P,\tilde{P}$ and therefore $G^k$ can be locally determined by the nodes. Moreover, since $(I_N-P)^{-1}\succeq I_N\succ P$, $(G^k)^{-1}=\rho (I_N-P)^{-1}\otimes I_d\succ \rho H$. Also, it can be verified that all the other conditions on $G^k$ required by DAMM-SQ hold. With this particular $G^k$, \eqref{eq:DAMMSQ} becomes
\begin{align*}
\mathbf{x}^{k+1} = &((P+(I_N-P)^2)\otimes I_d)\mathbf{x}^k\displaybreak[0]\\
&-\frac{1}{\rho}((I_N-P)\otimes I_d)(\nabla f(\mathbf{x}^k)+\mathbf{q}^k).
\end{align*}
Since $P\mathbf{1}=\mathbf{0}$ and $(\mathbf{1}^T\otimes I_d)\mathbf{q}^k=\mathbf{0}$, this indicates $\frac{1}{N}\sum_{i\in \mathcal{V}} x_i^{k+1}=\frac{1}{N}\sum_{i\in \mathcal{V}} x_i^k-\frac{1}{\rho}\Bigl(\frac{1}{N}\sum_{i\in \mathcal{V}} \nabla f_i(x_i^k)\Bigr)$. It can thus be seen that this example of DAMM-SQ tracks the average of all the local gradients so as to imitate the behavior of the (centralized) gradient descent method. 
\end{example}

\section{Existing Algorithms as Specializations}\label{sec:connection}

This section exemplifies that AMM and its distributed realizations generalize a variety of existing distributed optimization methods originally developed in different ways. %Although these methods were originally developed in different ways and some of them are not explicitly connected to the Method of Multipliers, we manage to unify them as special forms of DAMM-SQ/DAMM/AMM.

\subsection{Specializations of DAMM-SQ}\label{ssec:specialDAMMSQ}

DAMM-SQ described in Algorithm~\ref{alg:DAMMSQ} generalizes several distributed first-order and second-order algorithms for solving problem \eqref{eq:prob} with $h_i\equiv 0$ $\forall i\in \mathcal{V}$, including EXTRA \cite{ShiW15}, ID-FBBS \cite{Xu18}, and DQM \cite{Mokhtari16a}.

\subsubsection{EXTRA}\label{sssec:EXTRA}

EXTRA \cite{ShiW15} is a well-known first-order algorithm developed from a decentralized gradient descent method. From \cite[Eq. (3.5)]{ShiW15}, EXTRA can be expressed as
\begin{align}
\mathbf{x}^{k+1}\!=\!(\tilde{W}\!\otimes\! I_d)\mathbf{x}^k\!-\!\alpha\nabla f(\mathbf{x}^k)\!+\!\!\sum_{t=0}^k\!((W\!-\!\tilde{W})\!\otimes\! I_d)\mathbf{x}^t,\label{eq:EXTRA}
\end{align}
where $\mathbf{x}^0$ is arbitrarily given, $\alpha>0$, and $W, \tilde{W}\in \mathbb{R}^{N\times N}$ are two average matrices associated with $\mathcal{G}$\footnote{We say $W\in \mathbb{R}^{N\times N}$ is an average matrix associated with $\mathcal{G}$ if $W=W^T$, $W\mathbf{1}=\mathbf{1}$, $\|W-\mathbf{1}\mathbf{1}^T/N\|<1$, and $[W]_{ij}=0$ $\forall i\in \mathcal{V}$ $\forall j\notin \mathcal{N}_i\cup\{i\}$. It can be shown that $-I_N\prec W\preceq I_N$.}. By letting $\mathbf{q}^k = \frac{1}{\alpha}\sum_{t=0}^k ((\tilde{W}-W)\otimes I_d)\mathbf{x}^t$, \eqref{eq:EXTRA} becomes
\begin{align}
\mathbf{x}^{k+1} &= (\tilde{W}\otimes I_d)\mathbf{x}^k-\alpha\nabla f(\mathbf{x}^k)-\alpha \mathbf{q}^k,\label{eq:extraprimal}\\
\mathbf{q}^{k+1} &= \mathbf{q}^k+\frac{1}{\alpha}((\tilde{W}-W)\otimes I_d)\mathbf{x}^{k+1}.\label{eq:extradual}
\end{align}
This is in the form of DAMM-SQ with $\rho=1/\alpha$, $\tilde{P}=\tilde{W}-W$, $P=I_N-\tilde{W}$, and $G^k=\alpha I_{Nd}$. As \cite{ShiW15} assumes $\tilde{W}\succeq W$ and $\operatorname{Null}(\tilde{W}-W)=\operatorname{span}(\mathbf{1})$, Assumption~\ref{asm:weightmatrix} and \eqref{eq:q0} are guaranteed. Besides, \cite{ShiW15} assumes $\tilde{W}\succ \mathbf{O}$, so that $(G^k)^{-1}=\rho I_{Nd}\succeq \rho(I_N-\tilde{W})\otimes I_d=\rho H$. It is then straightforward to see that this particular $G^k$ satisfies all the conditions in Section~\ref{ssec:smooth}.

\subsubsection{ID-FBBS}

ID-FBBS \cite{Xu18} takes the form of \eqref{eq:extraprimal}--\eqref{eq:extradual}, except that $W=2\tilde{W}-I_N$ and $\mathbf{q}^0$ can be any vector in $S^\perp$. Since \cite{Xu18} also assumes $\tilde{W}\succ \mathbf{O}$, it follows from the analysis in Section~\ref{sssec:EXTRA} that ID-FBBS is a particular example of DAMM-SQ, where $\rho=1/\alpha$, $\tilde{P}=P=I_N-\tilde{W}$, and $G^k=\alpha I_{Nd}$, with Assumption~\ref{asm:weightmatrix} and all the conditions on $G^k$ in Section~\ref{ssec:smooth} satisfied.

\subsubsection{DQM}
DQM \cite{Mokhtari16a} is a distributed second-order method for solving problem \eqref{eq:prob} with strongly convex, smooth, and twice continuously differentiable $f_i$'s and zero $h_i$'s. DQM takes the following form: $x_i^{k+1} = x_i^k-(2c|\mathcal{N}_i|I_d+\nabla^2 f_i(x_i^k))^{-1}(c\!\sum_{j\in \mathcal{N}_i}\!(x_i^k\!-\!x_j^k)+\nabla f_i(x_i^k)\!+\!q_i^k)$ and $q_i^{k+1}=q_i^k+c\sum_{j\in \mathcal{N}_i}(x_i^{k+1}-x_j^{k+1})$, where $x_i^0$ $\forall i\in\mathcal{V}$ are arbitrarily given, $q_i^0$ $\forall i\in\mathcal{V}$ satisfy $\sum_{i\in \mathcal{V}} q_i^0 = \mathbf{0}$, and $c>0$. Observe that DAMM-SQ reduces to DQM if we set $G^k=(2c\operatorname{diag}(|\mathcal{N}_1|, \ldots, |\mathcal{N}_N|)\otimes I_d+\nabla^2 f(\mathbf{x}^k))^{-1}$, $\rho=c$, $p_{ij}=\tilde{p}_{ij}=-1$ $\forall \{i,j\}\in \mathcal{E}$, and $p_{ii}=\tilde{p}_{ii}=|\mathcal{N}_i|$ $\forall i\in \mathcal{V}$. Clearly, $P$ and $\tilde{P}$ satisfy Assumption~\ref{asm:weightmatrix}. Additionally, $(G^k)^{-1}\succeq 2c\operatorname{diag}(|\mathcal{N}_1|, \ldots, |\mathcal{N}_N|)\otimes I_d\succeq \rho P\otimes I_d=\rho H$, and $G^k$ meets all the other requirements in Section~\ref{ssec:smooth}.

\subsection{Specializations of DAMM}\label{ssec:specialDAMM}

A number of distributed algorithms for composite or nonsmooth convex optimization can be cast into the form of DAMM described in Algorithm~\ref{alg:mainalg}, including PGC \cite{Hong17}, PG-EXTRA \cite{ShiW15a}, DPGA \cite{Aybat18}, a decentralized ADMM in \cite{Shi14}, and D-FBBS \cite{Xu18}.

\subsubsection{PGC and PG-EXTRA}\label{sssec:PGC}

PGC \cite{Hong17} and PG-EXTRA \cite{ShiW15a} are two recently proposed distributed methods for solving problem \eqref{eq:prob}, where PGC is constructed upon ADMM \cite{Boyd11} and PG-EXTRA is an extension of EXTRA \cite{ShiW15} to address \eqref{eq:prob} with nonzero $h_i$'s. According to \cite[Section V-D]{Hong17}, PGC can be described by $\mathbf{q}^0\in S^\perp$, $\mathbf{q}^k = \mathbf{q}^0+\sum_{\ell=1}^k ((\Lambda_\beta(\tilde{W}-W))\otimes I_d)\mathbf{x}^{\ell}$ $\forall k\ge 1$, and $\mathbf{x}^{k+1}=\operatorname{\arg\;\min}_{\mathbf{x}\in \mathbb{R}^{Nd}}h(\mathbf{x})+\frac{1}{2}\|\mathbf{x}-(\tilde{W}\otimes I_d)\mathbf{x}^k\|_{\Lambda_\beta\otimes I_d}^2+\langle \nabla f(\mathbf{x}^k)+\mathbf{q}^k, \mathbf{x}\rangle$ $\forall k\ge 0$, where $\mathbf{x}^0$ is arbitrarily given and the parameters are chosen as follows: Let $\Lambda_\beta=\operatorname{diag}(\beta_1,\ldots,\beta_N)$ be a positive definite diagonal matrix and $W,\tilde{W}\in \mathbb{R}^{N\times N}$ be two row-stochastic matrices such that $\Lambda_\beta W$ and $\Lambda_\beta\tilde{W}$ are symmetric, $[W]_{ij},[\tilde{W}]_{ij}>0$ $\forall j\in \mathcal{N}_i\cup\{i\}$, and $[W]_{ij}=[\tilde{W}]_{ij}=0$ otherwise. To cast PGC in the form of DAMM, let $\rho = 1$, $\tilde{P}=\Lambda_\beta(\tilde{W}-W)$, and $P = \Lambda_\beta(I_N-\tilde{W})$. Then, starting from any $\mathbf{q}^0\in S^\perp$, the updates of PGC can be expressed as
\begin{align}
\mathbf{x}^{k+1} =&\operatorname{\arg\;\min}_{\mathbf{x}\in \mathbb{R}^{Nd}}h(\mathbf{x})\!+\!\frac{1}{2}\|\mathbf{x}-\mathbf{x}^k\|_{\Lambda_\beta \otimes I_d}^2\nonumber\displaybreak[0]\\
&+\langle \nabla f(\mathbf{x}^k)\!+\!\mathbf{q}^k\!+\!\rho(P\!\otimes\!I_d)\mathbf{x}^k, \mathbf{x}\rangle,\label{eq:PGextraprimal}\displaybreak[0]\\
\mathbf{q}^{k+1} =&\mathbf{q}^k+\rho(\tilde{P}\otimes I_d)\mathbf{x}^{k+1}.\label{eq:PGextradual}
\end{align}
This means that PGC is exactly in the form of DAMM with $\psi_i^k(\cdot)=\frac{\beta_i}{2}\|\cdot\|^2$. Note that $\Lambda_\beta\succeq \Lambda_\beta\tilde{W}$ and $\operatorname{Null}(\Lambda_\beta(I_N-\tilde{W}))=\operatorname{span}(\mathbf{1})$. In addition, \cite{Hong17} assumes $\Lambda_\beta\tilde{W}\succeq \Lambda_\beta W$, $\Lambda_\beta\tilde{W}\succeq \mathbf{O}$, and $\operatorname{Null}(\Lambda_\beta(\tilde{W}-W))=\operatorname{span}(\mathbf{1})$. Consequently, Assumption~\ref{asm:weightmatrix} and Assumption~\ref{asm:psi} hold. PG-EXTRA can also be described by \eqref{eq:PGextraprimal}--\eqref{eq:PGextradual} with $\beta_i=\rho>0$ $\forall i\in \mathcal{V}$ and $\mathbf{q}^0=\rho(\tilde{P}\otimes I_d)\mathbf{x}^0$, i.e., is a special form of PGC. Therefore, DAMM generalizes both PGC and PG-EXTRA.

\subsubsection{DPGA and a Decentralized ADMM}\label{ssec:DPGA}

DPGA \cite{Aybat18} is a distributed proximal gradient method and has the following form: Given arbitrary $x_i^0$ and $q_i^0=\mathbf{0}$,
\begin{align*}
x_i^{k+1}=&\operatorname{\arg\;\min}_{x\in\mathbb{R}^d}~h_i(x)+\frac{1}{2c_i}\|x-x_i^k\|^2\displaybreak[0]\\
&+\langle\nabla f_i(x_i^k)+q_i^k+\textstyle{\sum_{j\in\mathcal{N}_i\cup\{i\}}}\Gamma_{ij}x_j^k,x\rangle,\\
q_i^{k+1} =& q_i^k+\textstyle{\sum_{j\in \mathcal{N}_i\cup\{i\}}} \Gamma_{ij}x_j^{k+1},\quad\forall i\in\mathcal{V},
\end{align*}
where $c_i>0$ $\forall i\in \mathcal{V}$, $\Gamma_{ij}=\Gamma_{ji}<0$ $\forall \{i,j\}\in \mathcal{E}$, and $\Gamma_{ii}=-\sum_{j\in \mathcal{N}_i} \Gamma_{ij}$ $\forall i\in \mathcal{V}$. The above update equations of DPGA are equivalent to those of DAMM with $\psi_i^k(\cdot)=\frac{1}{2c_i}\|\cdot\|^2$, $\rho=1$, and $P,\tilde{P}$ such that $\tilde{p}_{ij},p_{ij}$ are equal to $\Gamma_{ij}$ if $\{i,j\}\in\mathcal{E}$ or $i=j$ and are equal to $0$ otherwise. Apparently, $P$ and $\tilde{P}$ satisfy Assumption~\ref{asm:weightmatrix}. Furthermore, due to the conditions on $c_i$ in \cite{Aybat18}, it is guaranteed that $\sum_{i\in \mathcal{V}}\psi_i^k(x_i)-\frac{\rho}{2}\|\mathbf{x}\|_H^2$ is convex and, thus, Assumption~\ref{asm:psi} holds. The decentralized ADMM in \cite{Shi14} for solving strongly convex and smooth problems can be viewed as a special case of DPGA with $c_i = \frac{1}{2c|\mathcal{N}_i|}$ $\forall i\in \mathcal{V}$ for some $c>0$ and $\Gamma_{ij}=\Gamma_{ji}=-c$ $\forall \{i,j\}\in \mathcal{E}$, so that it is also a specialization of DAMM.

\subsubsection{D-FBBS}
D-FBBS \cite{Xu18} addresses problem~\eqref{eq:prob} with $f_i\equiv 0$ $\forall i\in \mathcal{V}$, which is designed based on a Bregman method and an operator splitting technique. D-FBBS is described by \eqref{eq:PGextraprimal}--\eqref{eq:PGextradual} with $\beta_i=\rho>0$ $\forall i\in \mathcal{V}$, $\tilde{P}=P=I_N-W$ for an average matrix $W\in \mathbb{R}^{N\times N}$ associated with $\mathcal{G}$, $\mathbf{x}^0$ being arbitrary, and $\mathbf{q}^0\in S^\perp$. 
Clearly, $P,\tilde{P}$ satisfy Assumption~\ref{asm:weightmatrix}. Now let $\psi_i^k(\cdot)=\frac{\rho}{2}\|\cdot\|^2$, which satisfies Assumption~\ref{asm:psi} because \cite{Xu18} assumes $W\succ \mathbf{O}$. Therefore, we conclude that D-FBBS can be specialized from DAMM.

\subsection{Specializations of AMM}

Since DAMM-SQ and DAMM are subsets of AMM, the algorithms in Section~\ref{ssec:specialDAMMSQ} and Section~\ref{ssec:specialDAMM} are also specializations of AMM. Below, we present some other methods \cite{Makhdoumi17,Lei16,Nedic17,Mokhtari16} that can be specialized from AMM but do not belong to DAMM, DAMM-SC, or DAMM-SQ. 

\subsubsection{A Distributed ADMM}\label{ssec:distributedADMM}

In \cite{Makhdoumi17}, a distributed ADMM is proposed to solve \eqref{eq:prob} with $f_i\equiv 0$ $\forall i\in \mathcal{V}$: 
\begin{align*}%\label{eq:ADMM}
\mathbf{x}^{k+1} =& \operatorname{\arg\;\min}_{\mathbf{x}\in \mathbb{R}^{Nd}} h(\mathbf{x})+\langle Q^T\mathbf{w}^k, \mathbf{x}\rangle\displaybreak[0]\\
&+c\langle Q^T(\Lambda^{-1}\otimes I_d)Q\mathbf{x}^k,\mathbf{x}\rangle+\frac{c}{2}\|\mathbf{x}-\mathbf{x}^k\|_{\tilde{Q}}^2,\\
\mathbf{w}^{k+1} =& \mathbf{w}^k+c(\Lambda^{-1}\otimes I_d)Q\mathbf{x}^{k+1},
\end{align*}
where $\mathbf{x}^0\in\mathbb{R}^{Nd}$ is arbitrarily given and $\mathbf{w}^0=\mathbf{0}$. In the above, $c>0$, $\Lambda=\operatorname{diag}(|\mathcal{N}_1|+1, \ldots, |\mathcal{N}_N|+1)$, and $Q=\Gamma\otimes I_d$ with $\Gamma\in \mathbb{R}^{N\times N}$ satisfying $[\Gamma]_{ij}=0$ $\forall i\in\mathcal{V}$ $\forall j\notin \mathcal{N}_i\cup\{i\}$ and $\operatorname{Null}(\Gamma^T\Lambda^{-1}\Gamma)$ $=\operatorname{span}(\mathbf{1})$. Additionally, $\tilde{Q}=\operatorname{diag}(Q_1^TQ_1, \ldots, Q_N^TQ_N)\in \mathbb{R}^{Nd\times Nd}$, where $Q_i\in \mathbb{R}^{Nd\times d}$ $\forall i\in \mathcal{V}$ are such that $Q=(Q_1,\ldots,Q_N)$. It is shown in \cite{Makhdoumi17} that $\tilde{Q}\succeq Q^T(\Lambda^{-1}\!\otimes\!I_d)Q$. We may view this algorithm as a specialization from AMM \eqref{eq:ouralgqformprimal}--\eqref{eq:q0}, in which $\mathbf{q}^k=Q^T\mathbf{w}^k$, $H=\tilde{H}=Q^T(\Lambda^{-1}\otimes I_d)Q$, $\rho=c$, and $u^k(\mathbf{x})=\frac{\rho}{2}(\mathbf{x}-\mathbf{x}^k)^T(\tilde{Q}-H)(\mathbf{x}-\mathbf{x}^k)$. Apparently, $H,\tilde{H}$ are symmetric and positive semidefinite, and $\operatorname{Null}(H)=\operatorname{Null}(\tilde{H})=S$, as required in Section~\ref{ssec:AMM}. Also, since each $Q_i$ has full column rank, $\tilde{Q}\succ \mathbf{O}$. This guarantees Assumption~\ref{asm:ukassumption}.

\subsubsection{A Distributed Primal-Dual Algorithm}

In \cite{Lei16}, a distributed primal-dual algorithm is developed to solve \eqref{eq:prob} with each $h_i$ equal to the indicator function $\mathcal{I}_{X_i}$ with respect to a closed convex set $X_i$, and takes the following form: 
\begin{align}
&\!\!\mathbf{x}^{k+1}\!=\!P_X[\mathbf{x}^k\!\!-\!\alpha(\nabla f(\mathbf{x}^k)\!+\!(\Gamma\otimes I_d)\mathbf{w}^k\!\!+\!\!(\Gamma\otimes I_d)\mathbf{x}^k)],\label{eq:Lei16updateprimal}\\
&\!\!\mathbf{w}^{k+1}\!=\!\mathbf{w}^k+\alpha(\Gamma\otimes I_d)\mathbf{x}^k,\label{eq:Lei16updatedual}
\end{align}
where $\mathbf{x}^0,\mathbf{w}^0$ are arbitrarily given, $X=X_1\times\cdots\times X_N$, and $P_X[\cdot]$ is the projection onto $X$. Besides, $\Gamma\in \mathbb{R}^{N\times N}$ is a symmetric positive semidefinite matrix satisfying $[\Gamma]_{ij}=0$ $\forall i\in\mathcal{V}$ $\forall j\notin \mathcal{N}_i\cup\{i\}$ and $\operatorname{Null}(\Gamma)=\operatorname{span}(\mathbf{1})$, and $0<\alpha\le \frac{1}{2\|\Gamma\|}$. To see how this algorithm relates to AMM, we use \eqref{eq:Lei16updatedual} to rewrite \eqref{eq:Lei16updateprimal} as
\begin{align}
\mathbf{x}^{k+1}=&\operatorname{\arg\;\min}_{\mathbf{x}\in \mathbb{R}^{Nd}}~h(\mathbf{x})\nonumber\displaybreak[0]\\
&+\frac{1}{2\alpha}\|\mathbf{x}-\Bigl(\mathbf{x}^k-\alpha\bigl(\nabla f(\mathbf{x}^k)+(\Gamma\otimes I_d)\mathbf{w}^{k+1}\nonumber\displaybreak[0]\\
&+((\Gamma-\alpha \Gamma^2)\otimes I_d)\mathbf{x}^k\bigr)\Bigr)\|^2,\label{eq:Lei16updateprimalnew}
\end{align}
where $h(\mathbf{x})=\mathcal{I}_{X}(\mathbf{x})$. It can be seen that \eqref{eq:Lei16updatedual}--\eqref{eq:Lei16updateprimalnew} are equivalent to AMM \eqref{eq:AMMprimal}--\eqref{eq:AMMdual} with $\mathbf{v}^k=\mathbf{w}^{k+1}$, $\rho=\alpha$, $\tilde{H}=\Gamma^2\otimes I_d$, $H=(\Gamma/\alpha-\Gamma^2)\otimes I_d$, and $u^k(\mathbf{x})=\langle \nabla f(\mathbf{x}^k),\mathbf{x}\rangle+\frac{1}{2\alpha}\|\mathbf{x}-\mathbf{x}^k\|_{I_{Nd}-\alpha(\Gamma-\alpha \Gamma^2)\otimes I_d}^2$. Since $\alpha\le \frac{1}{2\|\Gamma\|}$, we have $H\succeq \Gamma\otimes I_d/(2\alpha)$ and $I_{Nd}-\alpha(\Gamma-\alpha \Gamma^2)\otimes I_d\succeq I_{Nd}-\alpha\Gamma\otimes I_d \succ \mathbf{O}$. Thus, $H,\tilde{H}$ satisfy the conditions required by AMM in Section~\ref{ssec:AMM}. Also, $u^k$ is strongly convex and satisfies Assumption~\ref{asm:ukassumption}.

\subsubsection{DIGing on Static Networks}

DIGing \cite{Nedic17} is a distributed gradient-tracking method for solving problem \eqref{eq:prob} with $h_i\equiv 0$ $\forall i\in \mathcal{V}$ over time-varying networks. Here, we only consider DIGing on static undirected networks. Let $\alpha>0$ and $W\in \mathbb{R}^{N\times N}$ satisfy $W\mathbf{1}=W^T\mathbf{1}=\mathbf{1}$, $[W]_{ij}=0$ $\forall i\in\mathcal{V}$ $\forall j\notin\mathcal{N}_i\cup\{i\}$, and $\|W-\frac{1}{N}\mathbf{1}\mathbf{1}^T\|<1$. It is shown in \cite{Nedic17} that DIGing with $W=W^T$ can be written as follows: 
\begin{align*}%\label{eq:DIGingequalform}
\mathbf{x}^{k+2}=&(2W\otimes I_d)\mathbf{x}^{k+1}-(W^2\otimes I_d)\mathbf{x}^k\displaybreak[0]\\
&-\alpha(\nabla f(\mathbf{x}^{k+1})-\nabla f(\mathbf{x}^k)),\quad\forall k\ge0,
\end{align*}
where $\mathbf{x}^0$ can be arbitrary and $\mathbf{x}^1=(W\otimes I_d)\mathbf{x}^0-\alpha\nabla f(\mathbf{x}^0)$. Adding the above equation from $k=0$ to $k=K-1$ yields
\begin{align*}
\mathbf{x}^{K+1} =& (W^2\otimes I_d)\mathbf{x}^K+((I_N-W)\otimes I_d)\mathbf{x}^0\displaybreak[0]\\
&-\sum_{k=0}^K ((I_N-W)^2\otimes I_d)\mathbf{x}^k-\alpha\nabla f(\mathbf{x}^K)
\end{align*}
for all $K\ge 0$. By letting $\mathbf{q}^0=\frac{1}{\alpha}((W^2-W)\otimes I_d)\mathbf{x}^0$, the above update is the same as
\begin{align*}
\mathbf{x}^{K+1} &= (W^2\otimes I_d)\mathbf{x}^K-\alpha\nabla f(\mathbf{x}^K)-\alpha \mathbf{q}^K,\displaybreak[0]\\
\mathbf{q}^{K+1} &= \mathbf{q}^K+\frac{1}{\alpha}((I_N-W)^2\otimes I_d)\mathbf{x}^{K+1}.
\end{align*}
Such an algorithmic form of DIGing is identical to AMM \eqref{eq:ouralgqformprimal}--\eqref{eq:q0} with the above given $\mathbf{q}^0\in S^\bot$, $\rho=1/\alpha$, $H=(I_N-W^2)\otimes I_d$, $\tilde{H}=(I_N-W)^2\otimes I_d$, and $u^k(\mathbf{x})=\langle \nabla f(\mathbf{x}^k), \mathbf{x}\rangle+\frac{\rho}{2}\|(W\otimes I_d)(\mathbf{x}-\mathbf{x}^k)\|^2$. It can be verified that $u^k$ $\forall k\ge 0$ and $H,\tilde{H}$ satisfy all the conditions in Section~\ref{ssec:AMM}. %In fact, the above special case of DIGing can be cast in the form of DAMM-SQ except that Assumption~\ref{asm:weightmatrix}(b) is not satisfied.

\iffalse
\subsubsection{PMM}
PMM \cite{Mokhtari16} is a combination of a proximal technique and the Method of Multipliers for solving problem \eqref{eq:prob} with $f_i$ $\forall i\in\mathcal{V}$ being strongly convex and smooth and with $h_i\equiv 0$ $\forall i\in\mathcal{V}$. Let $\alpha>0$, $\epsilon>0$, $W\in \mathbb{R}^{N\times N}$ be an average matrix associated with $\mathcal{G}$, and $\Gamma=(I_N-W)\otimes I_d$. Then, for any $\mathbf{x}^0,\mathbf{v}^0\in\mathbb{R}^{Nd}$, PMM is given by
\begin{equation}\label{eq:PMMupdates}
\begin{split}
\mathbf{x}^{k+1} &= \operatorname{\arg\;\min}_{\mathbf{x}\in \mathbb{R}^{Nd}} f(\mathbf{x})+(\mathbf{v}^k)^T\Gamma^{\frac{1}{2}}\mathbf{x}+\frac{\alpha}{2}\|\mathbf{x}\|_\Gamma^2+\frac{\epsilon}{2}\|\mathbf{x}-\mathbf{x}^k\|^2,\\
\mathbf{v}^{k+1} &= \mathbf{v}^k+\alpha \Gamma^{\frac{1}{2}}\mathbf{x}^{k+1}.
\end{split}
\end{equation}
It can be seen that AMM \eqref{eq:AMMprimal}--\eqref{eq:AMMdual} reduces to PMM when $\rho=\alpha$, $u^k(\mathbf{x}) = f(\mathbf{x})+\frac{\epsilon}{2}\|\mathbf{x}-\mathbf{x}^k\|^2$, and $\tilde{H}=H=\Gamma$. Since $\Gamma=\Gamma^T\succeq \mathbf{O}$ and $\operatorname{Null}(\Gamma)=S$, $H,\tilde{H}$ meet the conditions imposed in Section~\ref{ssec:AMM}. In addition, due to the assumptions on the $f_i$'s in \cite{Mokhtari16}, Assumption~\ref{asm:ukassumption} holds.
\fi

\subsubsection{ESOM}\label{sssec:ESOM}
ESOM \cite{Mokhtari16} is a class of distributed second-order algorithms that address problem \eqref{eq:prob} with $f_i$ $\forall i\in\mathcal{V}$ being strongly convex, smooth, and twice continuously differentiable and with $h_i\equiv 0$ $\forall i\in\mathcal{V}$. It is developed by incorporating a proximal technique and certain second-order approximations into the Method of Multipliers. To describe ESOM, let $\alpha>0$, $\epsilon>0$, and $W\in \mathbb{R}^{N\times N}$ be an average matrix associated with $\mathcal{G}$ such that $[W]_{ij}\ge 0$ $\forall i,j\in \mathcal{V}$. In addition, define $W_d:= \operatorname{diag}([W]_{11}, \ldots, [W]_{NN})$, $B:=\alpha(I_{Nd}+(W-2W_d)\otimes I_d)$, $D^k:=\nabla^2 f(\mathbf{x}^k)+\epsilon I_{Nd}+2\alpha(I_{Nd}-W_d\otimes I_d)$, and $Q^k(K) := (D^k)^{-\frac{1}{2}}\sum_{t=0}^K ((D^k)^{-\frac{1}{2}}B(D^k)^{-\frac{1}{2}})^t(D^k)^{-\frac{1}{2}}$, $K\ge0$. With the above notations, each ESOM-$K$ algorithm can be described by
\begin{align}
\mathbf{x}^{k+1}\!&=\! \mathbf{x}^k\!\!-\!Q^k(\!K\!)(\nabla f(\mathbf{x}^k)\!\!+\!\mathbf{q}^k\!+\!\alpha(I_{Nd}\!-\!W\!\!\otimes\!I_d)\mathbf{x}^k),\label{eq:ESOM1}\\
\mathbf{q}^{k+1}\!&= \mathbf{q}^k+\alpha(I_{Nd}-\!W\!\otimes I_d)\mathbf{x}^{k+1},\label{eq:ESOM2}
\end{align}
where $\mathbf{x}^0$ is any vector in $\mathbb{R}^{Nd}$ and $\mathbf{q}^0=\mathbf{0}$ which satisfies the initialization \eqref{eq:q0} of AMM. Note that the primal and dual updates of AMM, i.e., \eqref{eq:ouralgqformprimal}--\eqref{eq:ouralgqformdual}, reduce to \eqref{eq:ESOM1}--\eqref{eq:ESOM2} when $\tilde{H}=H=I_{Nd}-W\otimes I_d$, $\rho=\alpha$, and $u^k(\mathbf{x})=\frac{1}{2}\mathbf{x}^T((Q^k(K))^{-1}-\rho H)\mathbf{x}-\langle ((Q^k(K))^{-1}-\rho H)\mathbf{x}^k, \mathbf{x}\rangle+\langle \nabla f(\mathbf{x}^k), \mathbf{x}\rangle$. Clearly, $H$ and $\tilde{H}$ satisfy the conditions in Section~\ref{ssec:AMM}. Also, Assumption~\ref{asm:ukassumption} holds, since $(M+\epsilon+2\rho)I_{Nd}\succeq (Q^k(K))^{-1}\succeq \nabla^2 f(\mathbf{x}^k)+\epsilon I_{Nd}+\rho H$ \cite{Mokhtari16}, where $M>0$ is the Lipschitz constant of all the $\nabla f_i$'s. Note that unlike most specializations of AMM discussed in this section, $u^k(\cdot)+\frac{\rho}{2}\|\cdot\|_H^2$ for ESOM-$K$ with $K\ge1$ is non-separable.

\subsection{Connections between AMM and Existing Unifying Methods}

PUDA \cite{alghunaim20} and ABC \cite{Xu20} are two recently proposed distributed methods for convex composite optimization, which unify a number of existing methods including a subset of the aforementioned specializations of AMM. Nevertheless, unlike AMM that can be specialized to both first-order and second-order methods, PUDA and ABC are first-order algorithms. Moreover, AMM allows $h_i$ $\forall i\in\mathcal{V}$ to be non-identical, i.e., each node only needs to know a local portion of the global nonsmooth objective function $h$. In contrast, PUDA and ABC are restricted to the case of identical $h_i$'s, i.e., each node has to know the entire $h$, and it is not straightforward to extend their analyses to the more general case of non-identical $h_i$'s.

Although none of AMM, PUDA, and ABC can include the others as special cases, they are implicitly connected through the following algorithm: Let $\mathbf{q}^0=\mathbf{0}_{Nd}$. For any $k\ge 0$, let
\begin{align}
\mathbf{z}^{k+1} &= \operatorname{\arg\;\min}_{\mathbf{z}\in\mathbb{R}^{Nd}} u^k(\mathbf{z})+\frac{\rho}{2}\|\mathbf{z}\|_H^2+(\mathbf{q}^k)^T\mathbf{z},\label{eq:UDAABCz}\displaybreak[0]\\
\mathbf{x}^{k+1} &= \operatorname{\arg\;\min}_{\mathbf{x}\in\mathbb{R}^{Nd}}~h(\mathbf{x})+\frac{\rho}{2}\|\mathbf{x}-\mathbf{z}^{k+1}\|^2,\label{eq:UDAABCx}\\
\mathbf{q}^{k+1} &= \mathbf{q}^k+\rho \tilde{H}\mathbf{z}^{k+1}.\label{eq:UDAABCq}
\end{align}
Different from AMM whose primal update \eqref{eq:ouralgqformprimal} involves both the surrogate function and the nonsmooth objective, the above algorithm tackles these two parts separately and thus has two sequential primal minimization operations. Indeed, when $h\equiv 0$, \eqref{eq:UDAABCz}--\eqref{eq:UDAABCq} are identical to \eqref{eq:ouralgqformprimal}--\eqref{eq:ouralgqformdual}.

It can be shown that with particular $u^k$, $H$, and $\tilde{H}$, \eqref{eq:UDAABCz}--\eqref{eq:UDAABCq} and $\mathbf{q}^0=\mathbf{0}_{Nd}$ become equivalent to PUDA and ABC under some mild parameter conditions. Specifically, PUDA corresponds to $\mathbf{u}^k(\mathbf{z})=\frac{\rho}{2}\|\mathbf{z}-\mathbf{x}^k\|_{I-\mathcal{C}_1}^2+\langle \nabla \mathbf{f}(\mathbf{x}^k),\mathbf{z}\rangle$, $\rho>0$, $H=\mathcal{A}_1^{-1}-I_{Nd}+\mathcal{C}_1$, and $\tilde{H}=\mathcal{B}_1^2\mathcal{A}_1^{-1}$, while ABC corresponds to $\mathbf{u}^k(\mathbf{z})=\frac{\rho}{2}\|\mathbf{z}-\mathbf{x}^k\|_{\mathcal{B}_2^{-1}\mathcal{A}_2}^2+\langle \nabla \mathbf{f}(\mathbf{x}^k),\mathbf{z}\rangle$, $\rho>0$, $H=\mathcal{B}_2^{-1}(I_{Nd}-\mathcal{A}_2)$, and $\tilde{H}=\mathcal{B}_2^{-1}\mathcal{C}_2$. Here, $\mathcal{A}_1,\mathcal{B}_1,\mathcal{C}_1,\mathcal{A}_2,\mathcal{B}_2, \mathcal{C}_2\in\mathbb{R}^{Nd\times Nd}$ are proper matrices whose detailed designs can be found in \cite{alghunaim20,Xu20}. Such equivalence requires additional conditions that $\mathcal{A}_1$ and $\mathcal{B}_2$ are invertible, $\mathcal{C}_1\preceq I_{Nd}$, and $\mathcal{B}_1^2\mathcal{A}_1^{-1}$ is symmetric. In fact, these additional conditions are satisfied by many specializations of PUDA and ABC such as those in \cite[Table I]{alghunaim20} and \cite[Table II]{Xu20} when they choose positive definite average matrices. Also, the above options of $u^k$, $H$, and $\tilde{H}$ for PUDA and ABC satisfy our conditions for AMM in Section~\ref{ssec:AMM}.

To summarize, PUDA and ABC with some additional parameter conditions can be specialized from \eqref{eq:UDAABCz}--\eqref{eq:UDAABCq}, while AMM and \eqref{eq:UDAABCz}--\eqref{eq:UDAABCq} are the same if problem~\eqref{eq:prob} is smooth and are different otherwise. As a result, when solving smooth problems, AMM, PUDA, and ABC share a number of common special cases (e.g., EXTRA \cite{ShiW15} and DIGing \cite{Nedic17}).

% Only three settings of $\mathcal{A}$, $\mathcal{B}$, $\mathcal{C}$ in \cite[Table I]{alghunaim20} and \cite[Table II]{Xu20} not necessarily meet the requirements on $H$, $\tilde{H}$, and $\mathbf{u}^k$ in Section \ref{sec:DAMM}, which are based on an average matrix $W$ associated with the underlying graph $\mathcal{G}$ and 

\section{Convergence Analysis}\label{sec:convergence}

In this section, we analyze the convergence performance of AMM described by \eqref{eq:AMMprimal}--\eqref{eq:AMMdual} (which are equivalent to \eqref{eq:ouralgqformprimal}--\eqref{eq:q0}).

For the purpose of analysis, we add $H\succeq \tilde{H}$ to the conditions on $H,\tilde{H}$ in Section~\ref{ssec:AMM}, leading to Assumption~\ref{asm:tildeHSmallerthanH}.

\begin{assumption}\label{asm:tildeHSmallerthanH}
The matrices $H$ and $\tilde{H}$ are symmetric and $\operatorname{Null}(H)=\operatorname{Null}(\tilde{H})=S$, where $S$ is defined in \eqref{eq:equivprob}. In addition, $H\succeq \tilde{H}\succeq \mathbf{O}$.
\end{assumption}

In DAMM, DAMM-SC, and DAMM-SQ, we adopt $H=P\otimes I_d$ and $\tilde{H}=\tilde{P}\otimes I_d$, where $P$ and $\tilde{P}$ comply with Assumption~\ref{asm:weightmatrix}. Thus, as long as we further impose $\tilde{P}\preceq P$, Assumption~\ref{asm:tildeHSmallerthanH} holds. Moreover, all the existing specializations of AMM, DAMM, and DAMM-SQ in Section~\ref{sec:connection}, except DIGing \cite{Nedic17}, also need to satisfy Assumption~\ref{asm:tildeHSmallerthanH}. DIGing is not necessarily required to meet $H\succeq \tilde{H}$, yet it can easily satisfy $H\succeq \tilde{H}$ by letting its average matrix be symmetric positive semidefinite. 

In what follows, we let $(\mathbf{x}^k,\mathbf{v}^k)$ be generated by \eqref{eq:AMMprimal}--\eqref{eq:AMMdual}, and let $(\mathbf{x}^\star,\mathbf{v}^\star)$ be a primal-dual optimal solution pair of problem~\eqref{eq:eqprob}, which satisfies \eqref{eq:zgs_opt}. Also, let $\Lambda_M:=\operatorname{diag}(M_1, \ldots, M_N)\otimes I_d$, where $M_i>0$ is the Lipschitz constant of $\nabla f_i$. In addition, define
\begin{align*}
A^k:=\int_0^1 \nabla^2 u^k((1-s)\mathbf{x}^k+s\mathbf{x}^{k+1}) ds,\quad\forall k\ge 0.
\end{align*}
Note that every $A^k$ exists and is symmetric due to Assumption~\ref{asm:ukassumption}. Moreover, since $\nabla u^k(\mathbf{x}^k)=\nabla f(\mathbf{x}^k)$,
\begin{align}
&\nabla u^k(\mathbf{x}^{k+1})\nonumber\displaybreak[0]\\
=&\nabla u^k(\mathbf{x}^k)+\int_0^1 \nabla^2 u^k((1-s)\mathbf{x}^k+s\mathbf{x}^{k+1})(\mathbf{x}^{k+1}-\mathbf{x}^k)ds\nonumber\displaybreak[0]\\
=&\nabla f(\mathbf{x}^k)+A^k(\mathbf{x}^{k+1}-\mathbf{x}^k).\label{eq:meanvaluetheo}
\end{align}
Then, we introduce the following auxiliary lemma for deriving the main results.

\begin{lemma}\label{lemma:vrelationforboundedz}
Suppose Assumption~\ref{asm:problem}, Assumption~\ref{asm:ukassumption}, and Assumption~\ref{asm:tildeHSmallerthanH} hold. For any $\eta\in [0,1)$ and $k\ge 0$,
\begin{align}
&\frac{1}{\rho}\langle\mathbf{v}^k-\mathbf{v}^{k+1},\mathbf{v}^{k+1}-\mathbf{v}^\star\rangle \ge \eta\langle \mathbf{x}^k-\mathbf{x}^\star, \nabla f(\mathbf{x}^k)-\nabla f(\mathbf{x}^\star)\rangle\nonumber\displaybreak[0]\\
&+\langle\mathbf{x}^{k+1}-\mathbf{x}^\star,A^k(\mathbf{x}^{k+1}-\mathbf{x}^k)\rangle-\frac{\|\mathbf{x}^{k+1}-\mathbf{x}^k\|_{\Lambda_M}^2}{4(1-\eta)}.\label{eq:vrelationforboundedz}
\end{align}
\end{lemma}
\begin{proof}
	See Appendix \ref{ssec:proofvrelationforboundedz}.
\end{proof}

\subsection{Convergence under General Convexity}\label{ssec:sublinear}

In this subsection, we provide the convergence rates of AMM in solving the general convex problem \eqref{eq:prob}.

Let $\bar{\mathbf{x}}^k\!=\!\frac{1}{k}\sum_{t=1}^k \mathbf{x}^t$ $\forall k\ge 1$ be the running average of $\mathbf{x}^t$ from $t=1$ to $t=k$. Below, we derive sublinear convergence rates for (i) the consensus error $\|\tilde{H}^{\frac{1}{2}}\bar{\mathbf{x}}^k\|$, which represents the infeasibility of $\bar{\mathbf{x}}^k$ in solving the equivalent problem~\eqref{eq:eqprob}, and (ii) the objective error $|f(\bar{\mathbf{x}}^k)+h(\bar{\mathbf{x}}^k)-f(\mathbf{x}^\star)-h(\mathbf{x}^\star)|$, which is a direct measure of optimality. Throughout Section~\ref{ssec:sublinear}, we tentatively consider the following type of surrogate function that fulfills Assumption~\ref{asm:ukassumption}:
\begin{equation}\label{eq:constantmatrix}
u^k(\mathbf{x})=\frac{1}{2}\|\mathbf{x}-\mathbf{x}^k\|_A^2+\langle \nabla f(\mathbf{x}^k), \mathbf{x}\rangle, \quad\forall k\ge 0,
\end{equation}
where $A\in \mathbb{R}^{Nd\times Nd}$ satisfies $A=A^T$, $A\succeq \mathbf{O}$, and $A+\rho H\succ \mathbf{O}$. Such a choice of $u^k$ results in $A^k=A$ $\forall k\ge 0$. 

Note that AMM endowed with \eqref{eq:constantmatrix} still generalizes most existing algorithms discussed in Section~\ref{sec:connection}, including EXTRA \cite{ShiW15}, ID-FBBS \cite{Xu18}, PGC \cite{Hong17}, PG-EXTRA \cite{ShiW15a}, DPGA \cite{Aybat18}, the decentralized ADMM in \cite{Shi14}, D-FBBS \cite{Xu18}, the distributed ADMM in \cite{Makhdoumi17}, the distributed primal-dual algorithm in \cite{Lei16}, and DIGing \cite{Nedic17}. Although DQM \cite{Mokhtari16a} and ESOM \cite{Mokhtari16} are specialized from AMM with different $u^k$'s other than \eqref{eq:constantmatrix}, they all require problem~\eqref{eq:prob} to be strongly convex and smooth---In such a case, we will provide convergence rates for the general form of AMM in Section~\ref{ssec:linear}.

Now we bound $\|\tilde{H}^{\frac{1}{2}}\bar{\mathbf{x}}^k\|$ and $|f(\bar{\mathbf{x}}^k)+h(\bar{\mathbf{x}}^k)-f(\mathbf{x}^\star)-h(\mathbf{x}^\star)|$. This plays a key role in acquiring the rates at which these errors vanish.

\begin{lemma}\label{lemma:basicfuncval}
	Suppose Assumption~\ref{asm:problem}, Assumption~\ref{asm:tildeHSmallerthanH}, and \eqref{eq:constantmatrix} hold. For any $k\ge 1$,
	\begin{align}
	&\|\tilde{H}^{\frac{1}{2}}\bar{\mathbf{x}}^k\|=\frac{1}{\rho k}\|\mathbf{v}^k-\mathbf{v}^0\|,\label{eq:consensuserrornonstronglyconvex}\displaybreak[0]\\
	&f(\bar{\mathbf{x}}^k)\!+\!h(\bar{\mathbf{x}}^k)\!-\!f(\mathbf{x}^\star)\!-\!h(\mathbf{x}^\star)\nonumber\displaybreak[0]\\
	&\le\frac{1}{k}\!\Bigl(\frac{\|\mathbf{v}^0\|^2}{2\rho}\!+\!\frac{\|\mathbf{x}^0\!-\!\mathbf{x}^\star\|_A^2}{2}\!+\!\sum_{t=0}^{k-1}\!\frac{\|\mathbf{x}^{t+1}\!-\!\mathbf{x}^t\|_{\Lambda_M\!-\!A}^2}{2}\Bigr),\label{eq:general_funcvalle}\displaybreak[0]\\
	&f(\bar{\mathbf{x}}^k)\!+\!h(\bar{\mathbf{x}}^k)\!-\!f(\mathbf{x}^\star)\!-\!h(\mathbf{x}^\star)\!\ge -\!\frac{\|\mathbf{v}^\star\|\cdot\|\mathbf{v}^k-\mathbf{v}^0\|}{\rho k}.\label{eq:general_funcvalge}
	\end{align}
\end{lemma}
\begin{proof}
	See Appendix \ref{ssec:proofofbasicfuncval}.
\end{proof}

Observe from Lemma~\ref{lemma:basicfuncval} that as long as $\|\mathbf{v}^k-\mathbf{v}^0\|$ and $\sum_{t=0}^{k-1} \|\mathbf{x}^{t+1}-\mathbf{x}^t\|_{\Lambda_M-A}^2$ are bounded, $\|\tilde{H}^{\frac{1}{2}}\bar{\mathbf{x}}^k\|$ and $|f(\bar{\mathbf{x}}^k)+h(\bar{\mathbf{x}}^k)-f(\mathbf{x}^\star)-h(\mathbf{x}^\star)|$ are guaranteed to converge to $0$. The following lemma assures this is true with the help of Lemma~\ref{lemma:vrelationforboundedz}, in which $\mathbf{z}^k:=((\mathbf{x}^k)^T,(\mathbf{v}^k)^T)^T$ $\forall k\ge 0$, $\mathbf{z}^\star:=((\mathbf{x}^\star)^T,(\mathbf{v}^\star)^T)^T$, and $G:=\operatorname{diag}(A,I_{Nd}/\rho)$. 

\begin{lemma}\label{lemma:upperboundconstantweight}
	Suppose all the conditions in Lemma~\ref{lemma:basicfuncval} hold. Also suppose that $A \succ \Lambda_M/2$. For any $k\ge 1$,
	\begin{align}
	&\|\mathbf{v}^k-\mathbf{v}^0\|\le \|\mathbf{v}^0-\mathbf{v}^\star\|+\sqrt{\rho}\|\mathbf{z}^0-\mathbf{z}^\star\|_G,\label{eq:vkv0upperboundconstant}\displaybreak[0]\\
	&\sum_{t=0}^{k-1} \|\mathbf{x}^{t+1}-\mathbf{x}^t\|_{\Lambda_M-A}^2\le \frac{\|\mathbf{z}^0-\mathbf{z}^\star\|_G^2}{1-\underline{\sigma}},\label{eq:xt1xtupperboundconstant}
	\end{align}
	where $\underline{\sigma}:=\min\{\sigma|~\sigma A\succeq \Lambda_M/2\}\in (0,1)$.
\end{lemma}
\begin{proof}
	See Appendix \ref{ssec:proofofupperboundconstantweight}.
\end{proof}

The following theorem results from Lemma~\ref{lemma:basicfuncval} and Lemma~\ref{lemma:upperboundconstantweight}, which provides $O(1/k)$ rates for both the consensus error and the objective error at the running average $\bar{\mathbf{x}}^k$.

\begin{theorem}\label{theo:convergence}
Suppose all the conditions in Lemma~\ref{lemma:upperboundconstantweight} hold. For any $k\ge 1$,
\begin{align*}
&\|\tilde{H}^{\frac{1}{2}}\bar{\mathbf{x}}^k\|\le \frac{\|\mathbf{v}^0-\mathbf{v}^\star\|+\sqrt{\rho}\|\mathbf{z}^0-\mathbf{z}^\star\|_G}{\rho k},\displaybreak[0]\\
&f(\bar{\mathbf{x}}^k)\!+\!h(\bar{\mathbf{x}}^k)\!-\!f(\mathbf{x}^\star)\!-\!h(\mathbf{x}^\star)\displaybreak[0]\\
&\le \frac{1}{2k}\left(\frac{\|\mathbf{v}^0\|^2}{\rho}\!+\!\|\mathbf{x}^0-\mathbf{x}^\star\|_A^2\!+\!\frac{\|\mathbf{z}^0-\mathbf{z}^\star\|_G^2}{1-\underline{\sigma}}\right),\displaybreak[0]\\
&f(\bar{\mathbf{x}}^k)+h(\bar{\mathbf{x}}^k)-f(\mathbf{x}^\star)-h(\mathbf{x}^\star)\displaybreak[0]\\
&\ge -\frac{\|\mathbf{v}^\star\|(\|\mathbf{v}^0-\mathbf{v}^\star\|+\sqrt{\rho}\|\mathbf{z}^0-\mathbf{z}^\star\|_G)}{\rho k}.%\label{eq:nonstronggeneral_funcvalge}
\end{align*}
\end{theorem}
\begin{proof}
Substitute \eqref{eq:vkv0upperboundconstant}--\eqref{eq:xt1xtupperboundconstant} into \eqref{eq:consensuserrornonstronglyconvex}--\eqref{eq:general_funcvalge}.
\end{proof}

Among the existing distributed algorithms that are able to solve \emph{nonsmooth} convex optimization problems with \emph{non-strongly} convex objective functions (e.g., \cite{Zeng17,LiuYH19,LiZ19,Makhdoumi17,Aybat18,ShiW15a,Xu18,Hong17,Bianchi16,Nedic15,Xi17,Scutari19,Lorenzo16,Notarnicola2020,Daneshmand19,Xu20}), the best convergence rates are of $O(1/k)$, achieved by \cite{Makhdoumi17,Aybat18,Hong17,ShiW15a,Xu18,LiZ19,Scutari19,Xu20}. Indeed, AMM with all the conditions in Theorem~\ref{theo:convergence} generalizes the algorithms in \cite{Makhdoumi17,Aybat18,Hong17,ShiW15a,Xu18} as well as their parameter conditions. Like Theorem~\ref{theo:convergence}, \cite{Makhdoumi17,Aybat18,Hong17,Xu20} achieve $O(1/k)$ rates for both the objective error and the consensus error, whereas \cite{ShiW15a,Xu18,LiZ19,Scutari19} reach $O(1/k)$ rates in terms of optimality residuals that implicitly reflect deviations from an optimality condition, and they only guarantee $O(1/\sqrt{k})$ rates for the consensus error. Also note that \cite{Scutari19,Xu20} only address problem~\eqref{eq:prob} with identical $h_i$'s. 

Theorem~\ref{theo:convergence} provides new convergence results for some existing algorithms generalized by AMM. In particular, the $O(1/k)$ rate in terms of the objective error is new to PG-EXTRA \cite{ShiW15a} and D-FBBS \cite{Xu18} for nonsmooth problems as well as EXTRA \cite{ShiW15} and ID-FBBS \cite{Xu18} for smooth problems, which only had $O(1/k)$ rates with respect to optimality residuals before. Moreover, Theorem~\ref{theo:convergence} improves their $O(1/\sqrt{k})$ rates in reaching consensus to $O(1/k)$. Furthermore, Theorem~\ref{theo:convergence} extends the $O(1/k)$ rate of the distributed primal-dual algorithm in \cite{Lei16} for constrained smooth convex optimization to general nonsmooth convex problems, and allows PGC \cite{Hong17} to establish its $O(1/k)$ rate without the originally required condition that each $h_i$ needs to have a compact domain. 

Finally, we illustrate how the sublinear rates in Theorem~\ref{theo:convergence} are influenced by the graph topology. For simplicity, set $\mathbf{v}^0=\mb{0}$ and $\tilde{H}\preceq I_{Nd}$. Also, let $g^\star\in \partial h(\mathbf{x}^\star)$ be such that $\nabla f(\mathbf{x}^\star)+g^\star\in S^\perp$, which exists because of the optimality condition $-\sum_{i\in\mathcal{V}} \nabla f_i(x^\star)\in (\partial\sum_{i\in\mathcal{V}} h_i)(x^\star)$, and let $\mb{v}^\star=-(\tilde{H}^{\frac{1}{2}})^\dag(\nabla f(\mathbf{x}^\star)+g^\star)$. It follows from \eqref{eq:zgs_opt} that $\mb{v}^\star$ is an optimal solution to the Lagrange dual of \eqref{eq:eqprob}. Then, $\|\mb{v}^0-\mb{v}^\star\|=\|\mb{v}^\star\|\le \frac{1}{\sqrt{\lambda_{\tilde{H}}}}\|\nabla f(\mathbf{x}^\star)+g^\star\|$, where $\lambda_{\tilde{H}}\in(0,1]$ is the smallest nonzero eigenvalue of $\tilde{H}$. Note that $\|\mb{z}^0-\mb{z}^\star\|_G=(\|\mb{x}^0-\mb{x}^\star\|_A^2+\frac{1}{\rho}\|\mb{v}^0-\mb{v}^\star\|^2)^{\frac{1}{2}}\le\|\mb{x}^0-\mb{x}^\star\|_A+\frac{1}{\sqrt{\rho}}\|\mb{v}^0-\mb{v}^\star\|$. Thus, by letting $\rho=\frac{1}{\sqrt{\lambda_{\tilde{H}}}}$, we have $|f(\bar{\mathbf{x}}^k)+h(\bar{\mathbf{x}}^k)-f(\mathbf{x}^\star)-h(\mathbf{x}^\star)|\le O(\frac{1}{k\sqrt{\lambda_{\tilde{H}}}})$. In addition, $\|P_{S^\perp}[\bar{\mb{x}}^k]\|\le\frac{1}{\sqrt{\lambda_{\tilde{H}}}}\|P_{S^\perp}[\bar{\mb{x}}^k]\|_{\tilde{H}}=\frac{1}{\sqrt{\lambda_{\tilde{H}}}}\|\tilde{H}^{\frac{1}{2}}\bar{\mathbf{x}}^k\|\le O(\frac{1}{k\sqrt{\lambda_{\tilde{H}}}})$. Here, $\|P_{S^\perp}[\bar{\mb{x}}^k]\|$ plays a similar role as $\|\tilde{H}^{\frac{1}{2}}\bar{\mathbf{x}}^k\|$ in quantifying the consensus error. Therefore, \emph{denser network connectivity}, which is often indicated by larger $\lambda_{\tilde{H}}$, \emph{can yield faster convergence} of both the objective error and the consensus error (measured by $\|P_{S^\perp}[\bar{\mb{x}}^k]\|$). 

Compared to the existing specializations of AMM that also have $O(1/k)$ rates in solving \eqref{eq:prob}, the above $O(1/(k\sqrt{\lambda_{\tilde{H}}}))$ rate has better dependence on $\lambda_{\tilde{H}}\in(0,1]$ than the $O(1/(k\lambda_{\tilde{H}}))$ rate of DPGA \cite{Aybat18} and the $O(1/(k\lambda_{\tilde{H}}^2))$ rate of the distributed ADMM \cite{Makhdoumi17}. The remaining specializations discussed in Section~\ref{sec:connection} have not been provided with similar dependence of their convergence rates on the graph topology.

%Most algorithms discussed in Section \ref{sec:connection} that have $O(1/k)$ convergence rates do not analyzed explicit effect of the network, except DPGA \cite{Aybat18}, the distributed ADMM \cite{Makhdoumi17}, and ABC \cite{Xu20}. Among them DPGA \cite{Aybat18} has $O(1/(k\lambda_{\tilde{H}}))$ and the distributed ADMM \cite{Makhdoumi17} achieves $O(1/(k\lambda_{\tilde{H}}^2))$, which are both worse than the $O(1/(k\sqrt{\lambda_{\tilde{H}}}))$ result of AMM, while ABC \cite{Xu20} also reaches the $O(1/(k\sqrt{\lambda_{\tilde{H}}}))$ complexity.}

\subsection{Convergence under Local Restricted Strong Convexity}\label{ssec:linear}

In this subsection, we additionally impose a problem assumption and acquire the convergence rates of AMM. Henceforth, we no longer restrict $u^k$ to the form of \eqref{eq:constantmatrix}.

\begin{assumption}\label{asm:stronglyconvex}
The function $\sum_{i\in \mathcal{V}} f_i(x)$ is locally restricted strongly convex with respect to $x^\star$, where $x^\star$ is an optimum of problem~\eqref{eq:prob}, i.e., for any convex and compact set $C\subset\mathbb{R}^d$ containing $x^\star$, $\sum_{i\in \mathcal{V}} f_i(x)$ is restricted strongly convex with respect to $x^\star$ on $C$ with convexity parameter $m_{C}>0$. 
\end{assumption}

Restricted strong convexity has been studied in the literature. For example, \cite{YuanK16,ShiW15} consider its global version and \cite{WuX20} considers the local version as Assumption~\ref{asm:stronglyconvex}. Under Assumption~\ref{asm:stronglyconvex}, problem~\eqref{eq:prob} has a unique optimum $x^\star$, and problem~\eqref{eq:eqprob} has a unique optimum $\mathbf{x}^\star=((x^\star)^T,\ldots,(x^\star)^T)^T\in S$ \cite{WuX20}. Furthermore, Assumption \ref{asm:stronglyconvex} is less restricted than the standard global strong convexity condition. For instance, the objective function of logistic regression \cite{Bach14} is not globally strongly convex but meets Assumption \ref{asm:stronglyconvex}.

\begin{proposition}\label{prop:strongconvex}
Suppose Assumption~\ref{asm:problem} holds. If there exists $\epsilon>0$ such that each $f_i$ is twice continuously differentiable on the ball $B(x^\star,\epsilon):=\{x|~\|x-x^\star\|\le \epsilon\}$ and $\sum_{i\in\mathcal{V}}\nabla^2 f_i(x^\star)$ is positive definite, then Assumption~\ref{asm:stronglyconvex} holds.
\end{proposition}
\begin{proof}
	See Appendix \ref{ssec:proofofpropstrongtildef}.
\end{proof}

Proposition~\ref{prop:strongconvex} provides a sufficient condition for Assumption~\ref{asm:stronglyconvex}. When each $f_i$ is twice continuously differentiable, this sufficient condition is much weaker than global strong convexity which requires $\sum_{i\in\mathcal{V}}\nabla^2 f_i(x)\succ\mathbf{O}$ $\forall x\in\mathbb{R}^d$. 

Note that Assumption~\ref{asm:stronglyconvex} is unable to assure any strong convexity of $f(\mathbf{x})=\sum_{i\in\mathcal{V}}f_i(x_i)$, yet it guarantees the following property on $f(\mathbf{x})+\frac{\rho}{4}\|\mathbf{x}\|_{\tilde{H}}^2$.

\begin{proposition}\label{prop:strongconvexity}\cite[Lemma 1]{WuX20}
Under Assumption~\ref{asm:tildeHSmallerthanH} and Assumption~\ref{asm:stronglyconvex}, $f(\mathbf{x})+\frac{\rho}{4}\|\mathbf{x}\|_{\tilde{H}}^2$ is locally restricted strongly convex with respect to $\mathbf{x}^\star$.
\end{proposition}

Subsequently, we construct a convex and compact set $\mathcal{C}\subset \mathbb{R}^{Nd}$ containing $\mathbf{x}^\star$, so that the restricted strong convexity in Proposition~\ref{prop:strongconvexity} takes into effect on $\mathcal{C}$, leading to a parameter condition that guarantees $\mathbf{x}^k\in \mathcal{C}$ $\forall k\ge 0$. To introduce $\mathcal{C}$, note from Assumption~\ref{asm:ukassumption} that there exist symmetric positive semidefinite matrices $A_\ell$, $A_u\in \mathbb{R}^{Nd\times Nd}$ such that for any $\mathbf{x}\in \mathbb{R}^{Nd}$ and $k\ge 0$, $A_\ell \preceq \nabla^2 u^k(\mathbf{x}) \preceq A_u$. Let $\Delta:=\|A_u-A_\ell\|\ge0$. Moreover, define $A_a=(A_\ell+A_u)/2$ and $\tilde{G}=\operatorname{diag}(A_a,I_{Nd}/\rho)$, which are also symmetric positive semidefinite. Then, let $\mathcal{C}:=\left\{\mathbf{x}|~\|\mathbf{x}-\mathbf{x}^\star\|_{A_a+\rho\tilde{H}}^2\le\|\mathbf{z}^0-\mathbf{z}^\star\|_{\tilde{G}}^2+\rho\|\mathbf{x}^0\|_{\tilde{H}}^2\right\}$, where $\mathbf{z}^0$ and $\mathbf{z}^\star$ are defined above Lemma~\ref{lemma:upperboundconstantweight}. From Assumption~\ref{asm:ukassumption}(b) and Assumption~\ref{asm:tildeHSmallerthanH}, $A_a+\rho\tilde{H}\succ\mathbf{O}$, so that $\mathcal{C}$ is convex and compact. Thus, from Proposition~\ref{prop:strongconvexity}, $f(\mathbf{x})+\frac{\rho}{4}\|\mathbf{x}\|_{\tilde{H}}^2$ is restricted strongly convex with respect to $\mathbf{x}^\star$ on $\mathcal{C}$ with convexity parameter $m_\rho\in(0,\infty)$. Lemma~\ref{lemma:levelset} is another consequence of Lemma~\ref{lemma:vrelationforboundedz}, showing that $\mathbf{x}^k$ stays identically in $\mathcal{C}$.

\begin{lemma}\label{lemma:levelset}
Under Assumption~\ref{asm:problem}, Assumption~\ref{asm:ukassumption}, Assumption~\ref{asm:tildeHSmallerthanH}, and Assumption~\ref{asm:stronglyconvex}, $\mathbf{x}^k\in \mathcal{C}$ $\forall k\ge 0$ provided that
\begin{align}
A_a\!\succ\!\!\left(\!\frac{\Delta^2}{8\eta m_\rho}\!+\!\Delta\right)I_{Nd}\!+\!\frac{\Lambda_M}{2(1\!-\!\eta)}\text{ for some $\eta\in (0,1)$},\label{eq:restrictedstrongstepsize}
\end{align}
where $\Lambda_M$ is defined above Lemma~\ref{lemma:vrelationforboundedz}.
\end{lemma}

\begin{proof}
	See Appendix \ref{ssec:proofoflevelset}.
\end{proof}

When $\Delta=0$ (which means $\nabla^2u^k$ is constant) or $\sum_{i\in \mathcal{V}} f_i(x)$ is globally restricted strongly convex (which means $m_\rho$ can be independent of $\mathcal{C}$), it is straightforward to find $u^k$ so that \eqref{eq:restrictedstrongstepsize} holds. Otherwise, $m_\rho$ depends on $\mathcal{C}$ and thus on $A_a$, so that both sides of \eqref{eq:restrictedstrongstepsize} involve $A_a$. With that being said, \eqref{eq:restrictedstrongstepsize} can always be satisfied by proper $u^k$'s. To see this, arbitrarily pick $\eta\in(0,1)$ and $\bar{\lambda}_a,\underline{\lambda}_a>0$ such that $\bar{\lambda}_a>\underline{\lambda}_a>\frac{M}{2(1-\eta)}$, where $M:=\max_{i\in \mathcal{V}} M_i$. If we choose $u^k$ such that the corresponding $A_a$ satisfies $\underline{\lambda}_aI_{Nd}\preceq A_a\preceq \bar{\lambda}_aI_{Nd}$, then $\mathcal{C}$ is a subset of $\mathcal{C}':=\{\mathbf{x}|~\underline{\lambda}_a\|\mathbf{x}-\mathbf{x}^\star\|^2\le\bar{\lambda}_a\|\mathbf{x}^0-\mathbf{x}^\star\|^2+\frac{1}{\rho}\|\mathbf{v}^0-\mathbf{v}^\star\|^2+\rho\|\mathbf{x}^0\|_{\tilde{H}}^2\}$. Let $m_\rho'$ be any convexity parameter of $f(\mathbf{x})+\frac{\rho\|\mathbf{x}\|_{\tilde{H}}^2}{4}$ on $\mathcal{C}'$. Clearly, $m_\rho'\in (0,m_\rho]$ and is independent of $A_a$. Then, the following suffices to satisfy \eqref{eq:restrictedstrongstepsize}:
\begin{equation}\label{eq:strongercondition}
\left(\frac{\Delta^2}{8\eta m_\rho'}+\Delta+\underline{\lambda}_a\right)I_{Nd}\preceq A_a\preceq  \bar{\lambda}_aI_{Nd}.
\end{equation}
The lower and upper bounds on $A_a$ in \eqref{eq:strongercondition} do not need to depend on $A_a$. Also, \eqref{eq:strongercondition} is well-posed, since $\frac{\Delta^2}{8\eta m_\rho'}+\Delta+\underline{\lambda}_a\le\bar{\lambda}_a$ holds for sufficiently small $\Delta>0$. Therefore, $u^k$ $\forall k\ge0$ with such $\Delta$ and with $A_a$ satisfying \eqref{eq:strongercondition} meet \eqref{eq:restrictedstrongstepsize}.
 
Next, we present the convergence rates of AMM in both optimality and feasibility under the local restricted strong convexity condition. We first consider the smooth case of \eqref{eq:prob} with each $h_i$ identically equal to $0$, and provide a linear convergence rate for $\|\mathbf{z}^k-\mathbf{z}^\star\|_{\tilde{G}}^2+\rho\|\mathbf{x}^k\|_{\tilde{H}}^2$, which quantifies the distance to primal optimality, dual optimality, and primal feasibility of \eqref{eq:eqprob}. In Theorem~\ref{theo:linearrate}, we force $\mathbf{v^\star}$ to satisfy not only \eqref{eq:zgs_opt} but $\mathbf{v}^\star-\mathbf{v}^0\in S^\perp$ as well. Such $\mathbf{v}^\star$ is a particular optimum to the Lagrange dual of \eqref{eq:eqprob}, and can be chosen as $\mathbf{v}^\star=\tilde{\mathbf{v}}+((\frac{1}{N}\sum_{i\in \mathcal{V}} (v_i^0-\tilde{v}_i))^T, \ldots, (\frac{1}{N}\sum_{i\in \mathcal{V}} (v_i^0-\tilde{v}_i))^T)^T$, where $\tilde{\mathbf{v}}$ is any Lagrange dual optimum of \eqref{eq:eqprob} and $\mathbf{v}^0=((v_1^0)^T,\ldots,(v_N^0)^T)^T$ is the initial dual iterate of AMM. 

\begin{theorem}\label{theo:linearrate}
Suppose all the conditions in Lemma~\ref{lemma:levelset} hold. Also suppose $h_i(x)\equiv 0$ $\forall i\in \mathcal{V}$. Then, there exists $\tilde{\delta}\in (0,1)$ such that for each $k\ge 0$,
\begin{align}
&\|\mathbf{z}^{k+1}-\mathbf{z}^\star\|_{\tilde{G}}^2+\rho\|\mathbf{x}^{k+1}\|_{\tilde{H}}^2\nonumber\displaybreak[0]\\
&\le (1-\tilde{\delta})(\|\mathbf{z}^k-\mathbf{z}^\star\|_{\tilde{G}}^2+\rho\|\mathbf{x}^k\|_{\tilde{H}}^2).\label{eq:theononcompositeresult}
\end{align}
Specifically, $\tilde{\delta}=\max\{\delta|~B_i(\delta)\succeq \mathbf{O}, \forall i\in \{1,2,3\}\}$ with
\begin{align*}
B_1(\delta)\!=&(2\eta m_\rho\!\!-\!\beta\Delta)I_{Nd}\!-\!\delta A_a\!\!-\!\delta(1\!+\!\theta_1)(1\!+\!\theta_2)\Lambda_M^2/\!(\rho\lambda_{\tilde{H}}\!),\displaybreak[0]\\
B_2(\delta)\!=&\rho(1-\delta-\eta)\tilde{H}-\rho\delta(1+1/\theta_1)H\tilde{H}^\dag H,\displaybreak[0]\\
B_3(\delta)\!=& (1\!-\!\sigma)\!A_a\!\!-\!2\delta(1\!+\!\theta_1\!)(1\!+\!\!1/\theta_2\!)(\tfrac{A_u^2}{\rho\lambda_{\tilde{H}}}\!+\!\rho\bar{\lambda} H^{\frac{1}{2}}\tilde{H}^\dag H^{\frac{1}{2}}),
\end{align*}
where $\theta_1,\theta_2$ are arbitrary positive scalars, $\bar{\lambda}$ is any upper bound on $\|H\|$, $\eta\in (0,1)$ is given in \eqref{eq:restrictedstrongstepsize}, $\beta>0$ and $\sigma\in (0,1)$ are such that
\begin{equation}\label{eq:sigmabetarange}
\beta\Delta<2\eta m_\rho,\quad \sigma A_a\succeq (\frac{1}{4\beta}+1)\Delta I_{Nd}+\frac{\Lambda_M}{2(1-\eta)},
\end{equation}
and $\lambda_{\tilde{H}}>0$ is the smallest nonzero eigenvalue of $\tilde{H}$.
\end{theorem}

\begin{proof}
See Appendix \ref{ssec:proofoflinearrate}.
\end{proof}

In comparison with many existing works such as \cite{Shi14,Mokhtari16,Mokhtari16a,Makhdoumi17,Nedic17,QuG18,LiuYH19,LiZ19,PuS20,Xin20,Xu20,alghunaim20} that assume \emph{global} strong convexity to derive linear convergence rates, the linear rate in Theorem~\ref{theo:linearrate} only requires the weaker condition of \emph{local restricted} strong convexity in Assumption~\ref{asm:stronglyconvex}. Moreover, Theorem~\ref{theo:linearrate} provides new convergence results to a number of existing algorithms that can be specialized from AMM. Specifically, Theorem~\ref{theo:linearrate} establishes linear convergence for D-FBBS \cite{Xu18}, ID-FBBS \cite{Xu18}, and DPGA \cite{Aybat18}, which has never been discussed before. In addition, Theorem~\ref{theo:linearrate} allows EXTRA \cite{ShiW15}, DIGing \cite{Nedic17}, ESOM \cite{Mokhtari16}, DQM \cite{Mokhtari16a}, the distributed ADMMs \cite{Makhdoumi17,Shi14}, PUDA \cite{alghunaim20}, and ABC \cite{Xu20} to achieve linear convergence under less restrictive problem assumptions, including relaxing the global strong convexity in \cite{Nedic17,Mokhtari16,Mokhtari16a,Makhdoumi17,Shi14,Xu20,alghunaim20} and the global restricted strong convexity in \cite{ShiW15} to local restricted strong convexity, and eliminating the Lipschitz continuity of $\nabla^2 f_i$ required in \cite{Mokhtari16,Mokhtari16a}.

%Although Theorem~\ref{theo:linearrate} gives an explicit form of the convergence rate $(1-\tilde{\delta})$, the impact of problem and network characteristics on $\tilde{\delta}$ is still unclear. This is mainly because AMM is a general framework and the restricted strong convexity condition in Theorem~\ref{theo:linearrate} only holds locally. Below, we simplify the problem by assuming each $f_i$ is globally $m$-strongly convex for some $m>0$ and analyze the dependence of $\tilde{\delta}$ on some problem and network parameters. 

Both the problem and the graph have impacts on the linear rate $(1-\tilde{\delta})$ in Theorem~\ref{theo:linearrate}. Their impacts become explicit when each $f_i$ is globally restricted strongly convex with respect to $x^\star$ with convexity parameter $m>0$. For simplicity, we set $M_i=M\ge m$ $\forall i\in\mathcal{V}$ and $m_\rho=m$. Also, choose $H=\tilde{H}$ and proper $u^k$ $\forall k\ge0$. %Then, pick proper $u^k$ (which determines $A_a$, $A_u$, and $\Delta$), $\rho$, as well as $\eta, \beta, \sigma$ satisfying \eqref{eq:sigmabetarange}. 
In the expressions of $B_i(\delta)$ $\forall i=1,2,3$, we pick any $\theta_2>0$ and then let $\theta_1$ be such that $\max\{\delta: B_2(\delta)\succeq \mathbf{O}\}>\max\{\delta: B_1(\delta)\succeq \mathbf{O}\}$. We first suppose $m$ increases. Note that $u^k$ still satisfies \eqref{eq:restrictedstrongstepsize} with the same $\eta$, while we allow for a larger $\beta$ and thus a smaller $\sigma$ such that \eqref{eq:sigmabetarange} holds. Accordingly, $\max\{\delta: B_i(\delta)\succeq \mathbf{O}\}$ $\forall i=1,3$ become larger and, therefore, so does $\tilde{\delta}$. Similarly, smaller $M$ yields larger $\tilde{\delta}$. Now suppose the graph topology changes so that $\lambda_{\tilde{H}}$ increases, which implies denser connectivity of $\mathcal{G}$. Since $\bar{\lambda}$ can be any upper bound on $\|H\|$, we can always assume it is unchanged with the graph topology (e.g., $\bar{\lambda}=N$ when $H=L_{\mathcal{G}}\otimes I_d$). Then, it can be seen that $\tilde{\delta}$ increases.

The above discussions suggest that \emph{smaller condition number $\kappa_f:= M/m$ or denser network connectivity may lead to faster convergence} of AMM in solving strongly convex and smooth problems. Indeed, for some particular forms of AMM, such dependence can be more explicit and is comparable to some prior results. For example, let $H=\tilde{H}=(I_N-W)\otimes I_d/4$ with an average matrix $W$ associated with $\mathcal{G}$, $\rho=M\sqrt{2/\lambda_{\tilde{H}}}$, and $u^k$ be such that $\Delta=0$ and $A_a=\frac{\rho(3I_N+W)\otimes I_d}{4}$. Thus, $\rho I_{Nd}\succeq A_a\succeq \frac{\rho}{2}I_{Nd}$, $\|H\|\le \bar{\lambda}=\frac{1}{2}$, and $\lambda_{\tilde{H}}=\frac{1-\lambda_2(W)}{4}\in (0,\frac{1}{2})$, where $\lambda_2(W)$ is the second largest eigenvalue of $W$. Accordingly, $\rho\ge 2M$ and $A_a\succeq MI_{Nd}$. We may pick $\eta=\frac{1}{4}$ and $\sigma=\frac{2}{3}$ so that \eqref{eq:sigmabetarange} holds. Then, by choosing $\theta_1=\theta_2=1$ in Theorem \ref{theo:linearrate}, we can show that %$\tilde{\delta} \ge\min\{\frac{\sqrt{2(1-\lambda_2(W))}}{24\kappa_f}, \frac{1-\lambda_2(W)}{24(5-\lambda_2(W))}\}$. Hence,
\begin{equation}\label{eq:tildedeltasparse}
\tilde{\delta}=O(\min\{\tfrac{\sqrt{1-\lambda_2(W)}}{\kappa_f}, 1-\lambda_2(W)\}).
\end{equation}
Most existing linearly convergent specializations of AMM have not revealed such explicit relations as \eqref{eq:tildedeltasparse}, except the following. The decentralized ADMM \cite{Shi14} has the same result as \eqref{eq:tildedeltasparse}, and DIGing \cite{Nedic17} gives $\tilde{\delta}=O(\frac{(1-\lambda_2(W))^2}{\kappa_f^{1.5}})$ when $W\succeq\mathbf{O}$, which is worse than \eqref{eq:tildedeltasparse}. The distributed ADMM \cite{Makhdoumi17} can reach $\tilde{\delta} = O(\frac{1-\lambda_2(W)}{\sqrt{\kappa_f}})$, which is better than \eqref{eq:tildedeltasparse}. This is probably because \eqref{eq:tildedeltasparse} is directly from our analysis for more general problem assumptions and algorithmic forms.

%PUDA \cite{alghunaim20} and ABC \cite{Xu20} can achieve $\tilde{\delta}=O(\min\{\frac{1}{\kappa_f}, 1-\lambda_2(W)\})$

Finally, we remove the condition of $h_i\equiv 0$ $\forall i\in\mathcal{V}$ in Theorem~\ref{theo:linearrate} and derive the following result.

\begin{theorem}\label{thm:sublinearstronglyconvex}
Suppose all the conditions in Lemma~\ref{lemma:levelset} hold. For any $k\ge 1$,
\begin{align}
&\!\|\tilde{H}^{\frac{1}{2}}\bar{\mathbf{x}}^k\|\le \frac{d_0+\|\mathbf{v}^0-\mathbf{v}^\star\|}{\rho k},\label{eq:consensuserrornonsmoothstrongconvex}\displaybreak[0]\\
&\!f(\bar{\mathbf{x}}^k)+h(\bar{\mathbf{x}}^k)-f(\mathbf{x}^\star)-h(\mathbf{x}^\star)\!\le\frac{1}{k}\nonumber\displaybreak[0]\\
&\!\cdot\!\Bigl(\!\frac{d_0^2}{\rho}\bigl(\dfrac{\Delta}{2(2\eta m_\rho\!-\!\beta\Delta)}\!+\!\dfrac{1\!-\!\eta}{1\!-\!\sigma}\bigr)\!\!+\!\!\frac{\|\mathbf{x}^0\!-\!\mathbf{x}^\star\|_{A_u}^2}{2}\!+\!\!\frac{\|\mathbf{v}^0\|^2}{2\rho}\!\Bigr),\label{eq:objectiveerrorupperboundnonsmoothstrongconvex}\displaybreak[0]\\
&\!f(\bar{\mathbf{x}}^k)\!\!+\!h(\bar{\mathbf{x}}^k)\!-\!\!f(\mathbf{x}^\star)\!-\!h(\mathbf{x}^\star)\!\!\ge\!\! -\frac{\|\mathbf{v}^\star\|(d_0\!+\!\!\|\mathbf{v}^0\!-\!\mathbf{v}^\star\|)}{\rho k},\label{eq:objectiveerrorlowerboundnonsmoothstrongconvex}
\end{align}
where $d_0=\sqrt{\rho\|\mathbf{z}^0-\mathbf{z}^\star\|_{\tilde{G}}^2+\rho^2\|\mathbf{x}^0\|_{\tilde{H}}^2}$, $\eta\in(0,1)$ is given in \eqref{eq:restrictedstrongstepsize}, and $\beta>0$ and $\sigma\in (0,1)$ are such that \eqref{eq:sigmabetarange} holds.
\end{theorem}

\begin{proof}
See Appendix \ref{ssec:proofofnonsmoothstronglyconvex}.
\end{proof}

The convergence rates in Theorem~\ref{thm:sublinearstronglyconvex} are of the same order as those in Theorem~\ref{theo:convergence}. Theorem~\ref{theo:convergence} considers a more general optimization problem, while Theorem~\ref{thm:sublinearstronglyconvex} allows for more general $u^k$'s. In the literature, \cite{Notarnicola17,Yuan18,Wu19,Zeng17} also deal with nonsmooth, strongly convex problems. Even under global strong convexity, \cite{Yuan18,Wu19} provide convergence rates slower than $O(1/k)$, and \cite{Notarnicola17} derives an $O(1/k)$ rate of convergence only to dual optimality. Although \cite{Zeng17} attains linear convergence, it considers a more restricted problem, which assumes $f_i$ to be globally restricted strongly convex and the subgradients of $h_i$ to be uniformly bounded.

\begin{rem}
Compared to the existing algorithms in Section~\ref{sec:connection} that can be viewed as specializations of AMM, AMM is able to achieve convergence rates of the same or even better order under identical or weaker problem conditions. Moreover, although the theoretical results in Section~\ref{sec:convergence} require additional conditions on the parameters of AMM besides those in Section~\ref{sec:DAMM}, these parameter conditions are not restrictive. Indeed, when AMM reduces to the algorithms in Section~\ref{sec:connection}, its parameter conditions in Theorem~\ref{theo:convergence} still generalize those in \cite{ShiW15,Xu18,Hong17,ShiW15a,Aybat18,Makhdoumi17,Lei16} and partially overlap those in \cite{Nedic17} for non-strongly convex problems. The parameter conditions in Theorem~\ref{theo:linearrate} are more general than those in \cite{ShiW15,Mokhtari16a}, different from but intersecting with those in \cite{Nedic17,Mokhtari16}, and admittedly, relatively limited compared to those in \cite{Makhdoumi17,Shi14} for strongly convex problems, yet on the premise that $u^k$ is specialized to the corresponding particular forms specified in Section~\ref{sec:connection}.
\end{rem}

\section{Numerical Examples}\label{sec:numerical}

% In this section, we compare the convergence performance of several state-of-the-art distributed optimization algorithms and a new particular DAMM in addressing a numerical example of convex composite optimization. 

%\xw{In this section, we first demonstrate the competitive convergence performance of DAMM for solving composite problems with non-identical nonsmooth local objectives and then compare it with two methods that are also designed based on surrogate functions.}

This section demonstrates the competent practical performance of DAMM via numerical examples.

\subsection{Non-identical Local Nonsmooth Objectives}\label{ssec:numconstrained}
Consider the following constrained $l_1$-regularized problem:
\begin{equation}\label{eq:simuprob}
\operatorname{minimize}_{x\in \cap_{i\in \mathcal{V}} X_i}~ \sum_{i\in \mathcal{V}} \left(\frac{1}{2}\|B_ix-b_i\|^2+\frac{1}{N}\|x\|_1\right),
\end{equation}
where each $B_i\in \mathbb{R}^{m\times d}$, $b_i\in \mathbb{R}^m$, and $X_i=\{x|~\|x-a_i\|\le \|a_i\|+1\}$ with $a_i\in \mathbb{R}^d$. Note that $\cap_{i\in \mathcal{V}}X_i$ contains $\mathbf{0}$ and is nonempty. We set $f_i(x)=\frac{1}{2}\|B_ix-b_i\|^2$ and $h_i(x)=\frac{1}{N}\|x\|_1+\mathcal{I}_{X_i}(x)$ $\forall i\in\mathcal{V}$. We choose $N=20$, $d=5$, and $m=3$. Besides, we randomly generate each $B_i$, $b_i$, and $a_i$. The graph $\mathcal{G}$ is also randomly generated and contains $26$ links.

The simulation involves PG-EXTRA \cite{ShiW15a}, D-FBBS \cite{Xu18}, DPGA \cite{Aybat18}, and the distributed ADMM \cite{Makhdoumi17}, which are guaranteed to solve \eqref{eq:simuprob} with distinct $h_i$'s. From Section~\ref{sec:connection}, they are specializations of AMM, and the first three algorithms can also be specialized from DAMM. In addition to these existing methods, we run a particular DAMM with 
$\psi_i^k(\cdot)=\varpi_i(\cdot):=\frac{1}{2}\|\cdot\|_{B_i^TB_i+\epsilon I_d}^2$, $\epsilon>0$, which depends on the problem data. Note that this is a new algorithm for solving \eqref{eq:simuprob}. All these algorithms have similar computational complexities per iteration, while each iteration of the distributed ADMM requires twice the communication cost of the others. The algorithmic forms of PG-EXTRA, D-FBBS, DPGA, and the distributed ADMM are given in Section~\ref{sec:connection}. For DPGA, we let $c_i=c$ $\forall i\in \mathcal{V}$ for some $c>0$ and $\Gamma_{ij}=\frac{1}{2c}[M_{\mathcal{G}}]_{ij}$ $\forall i\in \mathcal{V}$ $\forall j\in \mathcal{N}_i\cup\{i\}$, where $M_{\mathcal{G}}$ is the Metropolis weight matrix defined below Assumption~\ref{asm:weightmatrix}. We also set the weight matrix $\Gamma=\frac{1}{2}M_{\mathcal{G}}$ in the distributed ADMM. We assign $P=\tilde{P}=\frac{1}{2}M_{\mathcal{G}}$ to the above new DAMM and to PG-EXTRA and D-FBBS when cast in the form of DAMM. The remaining parameters are all fine-tuned in their theoretical ranges to achieve the best possible performance.

\begin{figure}[tb]
	\centering
	\subfigure[Optimality error at $\bar{\mathbf{x}}^k$]{
		\includegraphics[width=0.65\linewidth,height=0.45\linewidth]{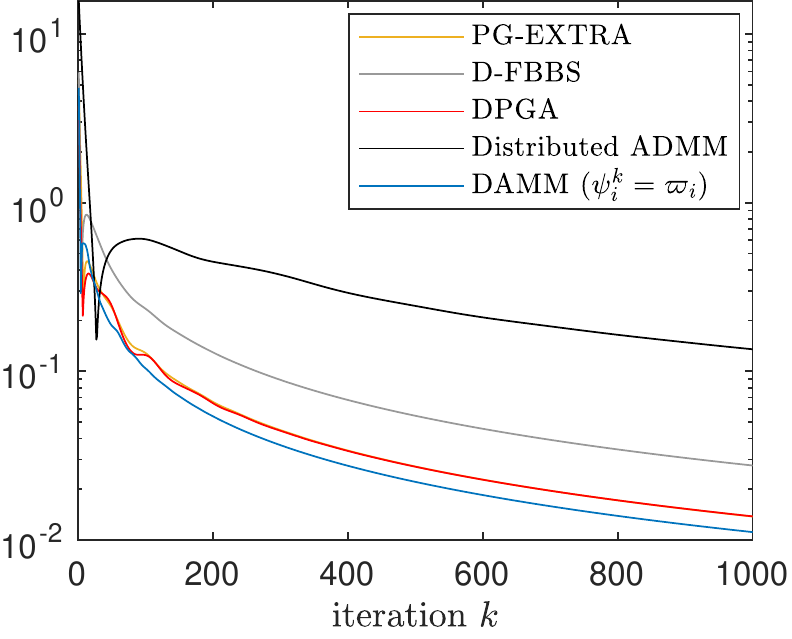}\label{fig:funcvalaverage}}
	\vfill
	\subfigure[Optimality error at $\mathbf{x}^k$]{
		\includegraphics[width=0.65\linewidth,height=0.45\linewidth]{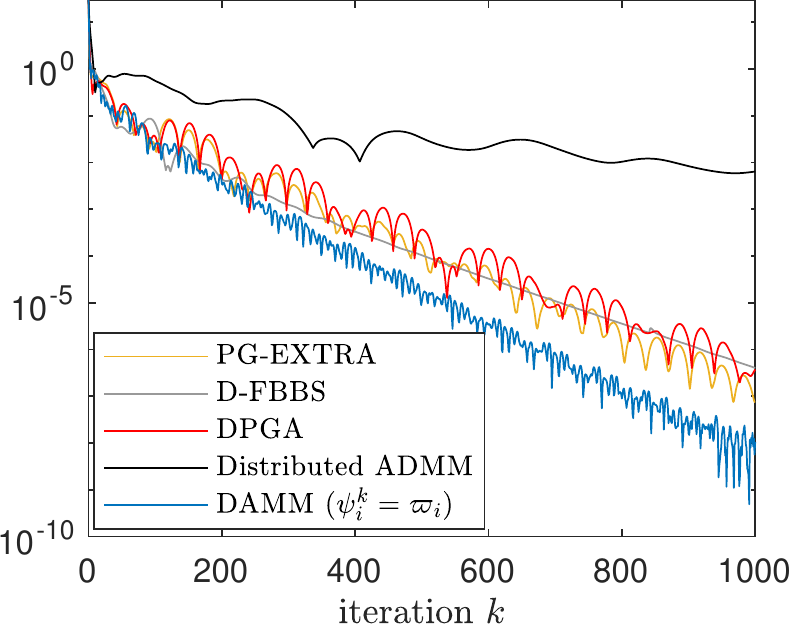}\label{fig:consensusaverage}}
	\caption{Convergence performance of PG-EXTRA, D-FBBS, DPGA, distributed ADMM, and DAMM with $\psi_i^k(\cdot)=\varpi_i(\cdot)$.}
	\label{fig:simu}
\end{figure}

Figure~\ref{fig:simu} plots the optimality error $|f(\cdot)+h(\cdot)-f(\mathbf{x}^\star)-h(\mathbf{x}^\star)|+\|\tilde{H}^{\frac{1}{2}}(\cdot)\|$ (which includes both the objective error and the consensus error) at the running average $\bar{\mathbf{x}}^k$ and the iterate $\mathbf{x}^k$, respectively. For each of the aforementioned algorithms, the running average produces a smoother curve, while the iterate converges faster. The new DAMM with $\psi_i^k(\cdot)=\varpi_i(\cdot)$ outperforms the other four methods, suggesting that AMM not only unifies a diversity of existing methods, but also creates novel distributed algorithms that achieve superior performance in solving various convex composite optimization problems.

\subsection{Comparison with Methods Using Surrogate Functions}\label{ssec:numunconstrained}

This subsection compares the convergence performance of the particular DAMM with $\psi_i^k(\cdot)=\varpi_i(\cdot)$ in Section~\ref{ssec:numconstrained} with that of NEXT \cite{Lorenzo16} and SONATA \cite{Scutari19}, which also utilize surrogate functions. However, NEXT and SONATA are only guaranteed to address problem~\eqref{eq:prob} with identical $h_i$'s. Thus, we change $X_i$ $\forall i\in\mathcal{V}$ in \eqref{eq:simuprob} to the unit ball $\{x|~\|x\|\le 1\}$. The remaining settings comply with Section~\ref{ssec:numconstrained}.

At each iteration, the computational complexities of all the three algorithms are roughly the same, yet the communication costs of NEXT and SONATA double that of the particular DAMM. For NEXT and SONATA, we follow the simulations in \cite{Lorenzo16,Scutari19} to approximate $f_i(x_i)$ by $f_i(x_i)+\frac{\tau}{2}\|x_i-x_i^k\|^2$ for some $\tau>0$ at each iteration $k$, and choose the average matrix as $I_N-M_{\mathcal{G}}$, where $M_{\mathcal{G}}$ is the Metropolis weight matrix. The diminishing step-size of NEXT is set as $c/k$ with $c>0$, and SONATA adopts a fixed step-size. Again, we fine-tune all the algorithm parameters within their theoretical ranges.

Figure~\ref{fig:unconstrainedsimu} displays the optimality error $|f(\mathbf{x}^k)+h(\mathbf{x}^k)-f(\mathbf{x}^\star)-h(\mathbf{x}^\star)|+\|\tilde{H}^{\frac{1}{2}}\mathbf{x}^k\|$ generated by the aforementioned algorithms. Observe that the particular DAMM converges much faster than the other two methods. The main reason for the slow convergence of SONATA in solving this numerical example is that its theoretical step-size is very small.

\begin{figure}[tb]
	\centering
	\includegraphics[width=0.65\linewidth,height=0.45\linewidth]{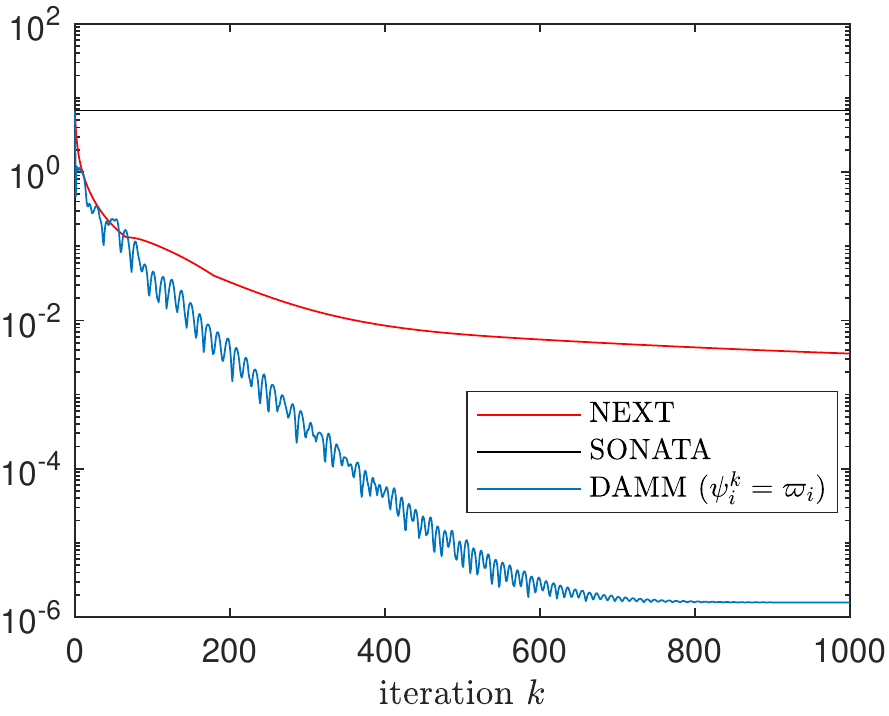}
	\caption{Convergence performance of NEXT, SONATA, and DAMM with $\psi_i^k(\cdot)=\varpi_i(\cdot)$.}
	\label{fig:unconstrainedsimu}
\end{figure}

\section{Conclusion}\label{sec:conclusion}

We have introduced a unifying Approximate Method of Multipliers (AMM) that emulates the Method of Multipliers via a surrogate function. Proper designs of the surrogate function lead to a wealth of distributed algorithms for solving convex composite optimization over undirected graphs. Sublinear and linear convergence rates for AMM are established under various mild conditions. The proposed AMM and its distributed realizations generalize a number of well-noted methods in the literature. Moreover, the theoretical convergence rates of AMM are no worse than and sometimes even better than the convergence results of such existing specializations of AMM in the sense of rate order, problem assumption, etc. The generality of AMM as well as its convergence analysis provides insights into the design and analysis of distributed primal-dual optimization algorithms, and allows us to explore high-performance specializations of AMM when addressing specific convex optimization problems in practice.

\appendix

%\subsection{Proof of Lemma~\ref{lem:inversefunction}}\label{ssec:proofofinversefunction}

%Since $\gamma^k$ is strongly convex and smooth, $\nabla \gamma^k$ is a bijective function and its inverse function $(\nabla \gamma^k)^{-1}$ exists. Arbitrarily pick $\tilde{\mathbf{x}}\in \mathbb{R}^{Nd}$ and let $\tilde{\mathbf{y}}=(\nabla \gamma^k)^{-1}(\tilde{\mathbf{x}})$. It follows from $\tilde{\mathbf{x}} = \nabla \gamma^k(\tilde{\mathbf{y}})$ and the first-order optimality condition that $\tilde{\mathbf{y}}=\operatorname{arg\;max}_{\mathbf{y}\in\mathbb{R}^{Nd}}\langle\tilde{\mathbf{x}},\mathbf{y}\rangle-\gamma^k(\mathbf{y})$. On the other hand, the Danskin's theorem \cite{Bertsekas99} suggests $\nabla (\gamma^k)^\star(\tilde{\mathbf{x}})=\operatorname{arg\;max}_{\mathbf{y}\in\mathbb{R}^{Nd}}\langle\tilde{\mathbf{x}},\mathbf{y}\rangle-\gamma^k(\mathbf{y})$. Therefore, $\nabla (\gamma^k)^\star(\tilde{\mathbf{x}})=\tilde{\mathbf{y}}=(\nabla \gamma^k)^{-1}(\tilde{\mathbf{x}})$.

\subsection{Proof of Lemma~\ref{lemma:vrelationforboundedz}}\label{ssec:proofvrelationforboundedz}

Define $g^{k+1}\!=\!-\nabla u^k(\mathbf{x}^{k+1})\!-\!\rho H\mathbf{x}^{k+1}\!-\!\tilde{H}^\frac{1}{2}\mathbf{v}^k$ and $g^\star\!=\!-\nabla f(\mathbf{x}^\star)\!-\!\tilde{H}^{\frac{1}{2}}\mathbf{v}^\star$. Due to \eqref{eq:AMMprimaloptimalitycondition} and \eqref{eq:zgs_opt}, $g^{k+1}\in\partial h(\mathbf{x}^{k+1})$ and $g^\star\in\partial h(\mathbf{x}^\star)$. It follows from \eqref{eq:meanvaluetheo} that
\begin{align}
\tilde{H}^{\frac{1}{2}}(\mathbf{v}^\star-\mathbf{v}^k)=&\nabla f(\mathbf{x}^k)-\nabla f(\mathbf{x}^\star)+A^k(\mathbf{x}^{k+1}-\mathbf{x}^k)\nonumber\displaybreak[0]\\
&+g^{k+1}-g^\star+\rho H\mathbf{x}^{k+1}.\label{eq:squareroottildeHvstarminusvk}
\end{align}
Also, since $\operatorname{Null}(\tilde{H}^{\frac{1}{2}})=\operatorname{Null}(\tilde{H})=S$, we have $\tilde{H}^{\frac{1}{2}}\mathbf{x}^\star=\mathbf{0}$. This, along with \eqref{eq:AMMdual}, gives $\mathbf{v}^k-\mathbf{v}^{k+1}=-\rho\tilde{H}^{\frac{1}{2}}(\mathbf{x}^{k+1}-\mathbf{x}^\star)$. Thus, $\langle \mathbf{v}^k-\mathbf{v}^{k+1}, \mathbf{v}^k-\mathbf{v}^\star\rangle=-\rho\langle \tilde{H}^{\frac{1}{2}}(\mathbf{x}^{k+1}-\mathbf{x}^\star), \mathbf{v}^k-\mathbf{v}^\star\rangle=\rho\langle \mathbf{x}^{k+1}-\mathbf{x}^\star, \tilde{H}^{\frac{1}{2}}(\mathbf{v}^\star-\mathbf{v}^k)\rangle$. By substituting \eqref{eq:squareroottildeHvstarminusvk} into the above equation and using $H\mathbf{x}^\star=\mathbf{0}$, we obtain $\langle \mathbf{v}^k-\mathbf{v}^{k+1}, \mathbf{v}^k-\mathbf{v}^\star\rangle=\rho\langle \mathbf{x}^{k+1}\!-\!\mathbf{x}^\star, \nabla f(\mathbf{x}^k)\!-\!\nabla f(\mathbf{x}^\star)\rangle\!+\!\rho\langle \mathbf{x}^{k+1}\!-\!\mathbf{x}^\star, g^{k+1}\!-\!g^\star\rangle
+\rho\langle \mathbf{x}^{k+1}-\mathbf{x}^\star, A^k(\mathbf{x}^{k+1}-\mathbf{x}^k)\rangle+\rho^2\|\mathbf{x}^{k+1}\|_H^2$. In this equation, $\langle \mathbf{x}^{k+1}-\mathbf{x}^\star, g^{k+1}-g^\star\rangle\ge 0$, because $g^{k+1}\in \partial h(\mathbf{x}^{k+1})$, $g^\star\in \partial h(\mathbf{x}^\star)$, and $h$ is convex. Moreover, due to $H\succeq \tilde{H}$ and \eqref{eq:AMMdual}, we have $\rho^2\|\mathbf{x}^{k+1}\|_H^2\ge \rho^2\|\mathbf{x}^{k+1}\|_{\tilde{H}}^2=\|\mathbf{v}^{k+1}-\mathbf{v}^k\|^2=\langle \mathbf{v}^k-\mathbf{v}^{k+1}, \mathbf{v}^k-\mathbf{v}^\star\rangle-\langle \mathbf{v}^k-\mathbf{v}^{k+1}, \mathbf{v}^{k+1}-\mathbf{v}^\star\rangle$. It follows that $\frac{1}{\rho}\langle\mathbf{v}^k\!-\!\mathbf{v}^{k+1},\mathbf{v}^{k+1}\!-\!\mathbf{v}^\star\rangle\ge\langle \mathbf{x}^{k+1}\!-\!\mathbf{x}^\star,\nabla f(\mathbf{x}^k)\!-\!\nabla f(\mathbf{x}^\star)\rangle+\langle \mathbf{x}^{k+1}-\mathbf{x}^\star,A^k(\mathbf{x}^{k+1}-\mathbf{x}^k)\rangle$. Hence, to prove \eqref{eq:vrelationforboundedz}, it suffices to show that for any $\eta\in(0,1)$,
\begin{align}
&\langle \mathbf{x}^{k+1}-\mathbf{x}^\star, \nabla f(\mathbf{x}^k)-\nabla f(\mathbf{x}^\star)\rangle\nonumber\displaybreak[0]\\
&\ge\!\eta\langle \mathbf{x}^k\!-\!\mathbf{x}^\star, \nabla f(\mathbf{x}^k)\!-\!\nabla f(\mathbf{x}^\star)\rangle\!-\!\frac{\|\mathbf{x}^{k+1}\!-\!\mathbf{x}^k\|_{\Lambda_M}^2}{4(1-\eta)}.\label{eq:xt1xstargrafxtgrafxstar}
\end{align}
To do so, note from the AM-GM inequality and the Lipschitz continuity of $\nabla f_i$ that for any $i\in \mathcal{V}$ and $c_i>0$, 
$\langle x_i^{k+1}-x_i^k, \nabla f_i(x_i^k)-\nabla f_i(x^\star)\rangle\ge  -c_i\|\nabla f_i(x_i^k)-\nabla f_i(x^\star)\|^2-\frac{1}{4c_i}\|x_i^{k+1}-x_i^k\|^2\ge -c_iM_i\langle x_i^k-x^\star, \nabla f_i(x_i^k)-\nabla f_i(x^\star)\rangle-\frac{1}{4c_i}\|x_i^{k+1}-x_i^k\|^2$. By adding the above inequality over all $i\in \mathcal{V}$ with $c_i=\frac{1-\eta}{M_i}$, we have $\langle \mathbf{x}^{k+1}-\mathbf{x}^k, \nabla f(\mathbf{x}^k)-\nabla f(\mathbf{x}^\star)\rangle+(1-\eta)\langle \mathbf{x}^k-\mathbf{x}^\star, \nabla f(\mathbf{x}^k)-\nabla f(\mathbf{x}^\star)\rangle\ge-\frac{\|\mathbf{x}^{k+1}-\mathbf{x}^k\|_{\Lambda_M}^2}{4(1-\eta)}$. Then, adding $\eta\langle \mathbf{x}^k-\mathbf{x}^\star, \nabla f(\mathbf{x}^k)-\nabla f(\mathbf{x}^\star)\rangle$ to both sides of this inequality leads to \eqref{eq:xt1xstargrafxtgrafxstar}.

\subsection{Proof of Lemma~\ref{lemma:basicfuncval}}\label{ssec:proofofbasicfuncval}

Let $k\ge1$. Due to \eqref{eq:AMMdual}, $\tilde{H}^{\frac{1}{2}}\bar{\mathbf{x}}^k = \frac{1}{k}\sum_{t=1}^k \tilde{H}^{\frac{1}{2}}\mathbf{x}^t=\frac{1}{\rho k}(\mathbf{v}^k-\mathbf{v}^0)$, so that \eqref{eq:consensuserrornonstronglyconvex} holds. Next, we prove \eqref{eq:general_funcvalle}. For simplicity, define $J^t(\mathbf{x})=\langle\nabla f(\mathbf{x}^t), \mathbf{x}\rangle+\|\mathbf{x}-\mathbf{x}^t\|_A^2/2+h(\mathbf{x})+\rho\|\mathbf{x}\|_H^2/2+(\mathbf{v}^t)^T\tilde{H}^{\frac{1}{2}}\mathbf{x}$ $\forall t\ge0$. Since $J^t(\mathbf{x})-\frac{\|\mathbf{x}\|_A^2}{2}$ is convex, $\bigl(J^t(\mathbf{x}^\star)-\frac{\|\mathbf{x}^\star\|_A^2}{2}\bigr)-\bigl(J^t(\mathbf{x}^{t+1})-\frac{\|\mathbf{x}^{t+1}\|_A^2}{2}\bigr)\ge\langle\tilde{g}^{t+1}\!-\!A\mathbf{x}^{t+1}, \mathbf{x}^\star\!-\!\mathbf{x}^{t+1}\rangle$ for all $\tilde{g}^{t+1}\!\in\!\partial J^t(\mathbf{x}^{t+1})$. Due to \eqref{eq:AMMprimal}, $\mathbf{0}\in \partial J^t(\mathbf{x}^{t+1})$. Thus, by letting $\tilde{g}^{t+1}=\mathbf{0}$ in the above inequality, we obtain 
\begin{equation}\label{eq:constantmatrixxminimize}
J^t(\mathbf{x}^{t+1})-J^t(\mathbf{x}^\star)\le-\|\mathbf{x}^{t+1}-\mathbf{x}^\star\|_A^2/2.
\end{equation}
In addition, because of \eqref{eq:AMMdual} and $H\succeq \tilde{H}$,
\begin{align}
&\langle \tilde{H}^{\frac{1}{2}}\mathbf{x}^{t+1}, \mathbf{v}^t\rangle = \frac{1}{\rho}\langle \mathbf{v}^{t+1}-\mathbf{v}^t, \mathbf{v}^t\rangle\nonumber\displaybreak[0]\\ 
=&(\|\mathbf{v}^{t+1}\|^2-\|\mathbf{v}^t\|^2-\|\mathbf{v}^{t+1}-\mathbf{v}^t\|^2)/(2\rho)\nonumber\displaybreak[0]\\
=&(\|\mathbf{v}^{t+1}\|^2-\|\mathbf{v}^t\|^2)/(2\rho)-\rho\|\mathbf{x}^{t+1}\|_{\tilde{H}}^2/2\nonumber\displaybreak[0]\\
\ge&(\|\mathbf{v}^{t+1}\|^2-\|\mathbf{v}^t\|^2)/(2\rho)-\rho\|\mathbf{x}^{t+1}\|_H^2/2.\label{eq:xk1tildeHxvk}
\end{align}
Also, due to the convexity and the $M_i$-smoothness of each $f_i$,
\begin{align}
& f(\mathbf{x}^{t+1})-f(\mathbf{x}^\star)= (f(\mathbf{x}^{t+1})-f(\mathbf{x}^t))+(f(\mathbf{x}^t)-f(\mathbf{x}^\star))\nonumber\displaybreak[0]\\
\le & \langle \nabla f(\mathbf{x}^t), \mathbf{x}^{t+1}\!\!-\!\mathbf{x}^t\rangle\!+\!\|\mathbf{x}^{t+1}\!\!-\!\mathbf{x}^t\|_{\Lambda_M}^2/2\!+\!\langle \nabla f(\mathbf{x}^t), \mathbf{x}^t\!-\!\mathbf{x}^\star\rangle\nonumber\displaybreak[0]\\
=   & \langle \nabla f(\mathbf{x}^t), \mathbf{x}^{t+1}-\mathbf{x}^\star\rangle+\|\mathbf{x}^{t+1}-\mathbf{x}^t\|_{\Lambda_M}^2/2.\label{eq:fxk1minusfstarupperbound}
\end{align}
Through combining \eqref{eq:constantmatrixxminimize}, \eqref{eq:xk1tildeHxvk}, and \eqref{eq:fxk1minusfstarupperbound} and utilizing $H\mathbf{x}^\star=\tilde{H}^{\frac{1}{2}}\mathbf{x}^\star = \mathbf{0}$, we derive 
\begin{align}
&f(\mathbf{x}^{t+1})+h(\mathbf{x}^{t+1})-f(\mathbf{x}^\star)-h(\mathbf{x}^\star)+\frac{\|\mathbf{v}^{t+1}\|^2-\|\mathbf{v}^t\|^2}{2\rho}\nonumber\displaybreak[0]\\
\le&\frac{\|\mathbf{x}^t\!-\!\mathbf{x}^\star\|_A^2}{2}\!-\!\frac{\|\mathbf{x}^{t+1}\!-\!\mathbf{x}^\star\|_A^2}{2}\!+\!\frac{\|\mathbf{x}^{t+1}\!-\!\mathbf{x}^t\|_{\Lambda_M-A}^2}{2}.\label{eq:nonstrongfplushoptimaldistance}
\end{align}
Now adding \eqref{eq:nonstrongfplushoptimaldistance} over $t=0,\ldots,k-1$ yields $\sum_{t=1}^k \left(f(\mathbf{x}^t)+h(\mathbf{x}^t)-f(\mathbf{x}^\star)-h(\mathbf{x}^\star)\right)\le\frac{\|\mathbf{v}^0\|^2}{2\rho}+\frac{\|\mathbf{x}^0-\mathbf{x}^\star\|_A^2}{2} +\sum_{t=0}^{k-1}\frac{\|\mathbf{x}^{t+1}-\mathbf{x}^t\|_{\Lambda_M-A}^2}{2}$. Moreover, the convexity of $f$ and $h$ implies $f(\bar{\mathbf{x}}^k)+h(\bar{\mathbf{x}}^k) \le \frac{1}{k}\sum_{t=1}^k \left(f(\mathbf{x}^t)+h(\mathbf{x}^t)\right)$. Combining the above results in \eqref{eq:general_funcvalle}.

Finally, to prove \eqref{eq:general_funcvalge}, note that $\mathbf{x}^\star$ is an optimal solution of $\min_{\mathbf{x}} f(\mathbf{x})+h(\mathbf{x})+\langle \mathbf{v}^\star, \tilde{H}^{\frac{1}{2}}\mathbf{x}\rangle$. It then follows from $\tilde{H}^{\frac{1}{2}}\mathbf{x}^\star = \mathbf{0}$ that $f(\mathbf{x}^\star)+h(\mathbf{x}^\star) \le f(\bar{\mathbf{x}}^t)+h(\bar{\mathbf{x}}^t)+\langle \mathbf{v}^\star, \tilde{H}^{\frac{1}{2}}\bar{\mathbf{x}}^t\rangle\le f(\bar{\mathbf{x}}^t)+h(\bar{\mathbf{x}}^t)+\|\mathbf{v}^\star\|\cdot \|\tilde{H}^{\frac{1}{2}}\bar{\mathbf{x}}^t\|$. This, together with \eqref{eq:consensuserrornonstronglyconvex}, implies \eqref{eq:general_funcvalge}.

\subsection{Proof of Lemma~\ref{lemma:upperboundconstantweight}}\label{ssec:proofofupperboundconstantweight}

From the definition of $G$, for each $t\ge0$,
\begin{align}
&\|\mathbf{z}^t-\mathbf{z}^\star\|_G^2-\|\mathbf{z}^{t+1}-\mathbf{z}^\star\|_G^2-\|\mathbf{z}^{t+1}-\mathbf{z}^t\|_G^2\nonumber\displaybreak[0]\\
=&2\langle G(\mathbf{z}^t-\mathbf{z}^{t+1}), \mathbf{z}^{t+1}-\mathbf{z}^\star\rangle\nonumber\displaybreak[0]\\
=&2\langle A(\mathbf{x}^t\!\!-\!\mathbf{x}^{t+1}), \mathbf{x}^{t+1}\!\!-\!\mathbf{x}^\star\rangle\!+\!\frac{2}{\rho}\langle \mathbf{v}^t\!\!-\!\mathbf{v}^{t+1}, \mathbf{v}^{t+1}\!\!-\!\mathbf{v}^\star\rangle.\label{eq:zkminuszstarexpansion}
\end{align}
By substituting \eqref{eq:vrelationforboundedz} with $A^t=A$ and $\eta=0$ into \eqref{eq:zkminuszstarexpansion}, 
\begin{align}
&\|\mathbf{z}^t-\mathbf{z}^\star\|_G^2-\|\mathbf{z}^{t+1}-\mathbf{z}^\star\|_G^2\ge\frac{1}{\rho}\|\mathbf{v}^{t+1}-\mathbf{v}^t\|^2\nonumber\displaybreak[0]\\
&+\|\mathbf{x}^{t+1}-\mathbf{x}^t\|_{A-\frac{\Lambda_M}{2}}^2\ge(1-\underline{\sigma})\|\mathbf{x}^{t+1}-\mathbf{x}^t\|_A^2,\label{eq:monotonezdistancegeneralconvex}
\end{align}
where the last step is due to $\underline{\sigma}A\succeq \Lambda_M/2$. Adding \eqref{eq:monotonezdistancegeneralconvex} over $t=0,\ldots,k-1$ and using $\|\mathbf{v}^k\!-\!\mathbf{v}^\star\|\!\le\!\sqrt{\rho}\|\mathbf{z}^k\!-\!\mathbf{z}^\star\|_G$, we have $\frac{1}{\rho}\|\mathbf{v}^k-\mathbf{v}^\star\|^2+(1-\underline{\sigma})\sum_{t=0}^{k-1}\|\mathbf{x}^{t+1}-\mathbf{x}^t\|_A^2 \le \|\mathbf{z}^0-\mathbf{z}^\star\|_G^2$. Therefore, \eqref{eq:vkv0upperboundconstant} can be proved by combining the above inequality with $\|\mathbf{v}^k-\mathbf{v}^0\|\le \|\mathbf{v}^k-\mathbf{v}^\star\|+\|\mathbf{v}^0-\mathbf{v}^\star\|$. Furthermore, the above inequality, along with $A\succ \Lambda_M-A$, implies \eqref{eq:xt1xtupperboundconstant}.

\subsection{Proof of Proposition \ref{prop:strongconvex}}\label{ssec:proofofpropstrongtildef}
	
For convenience, let $\tilde{f}(x)=\sum_{i\in\mathcal{V}} f_i(x)$. Since $\nabla^2 \tilde{f}(x^\star)\succ \mathbf{O}$ and $\nabla^2 \tilde{f}$ is continuous around $x^\star$, there exist $\epsilon_1\in (0,\epsilon]$ and $c>0$ such that $\nabla^2\tilde{f}(x)\succeq cI_d$ $\forall x\in B(x^\star,\epsilon_1)$. Therefore,
	\begin{equation}\label{eq:strongconvexitytildef}
		\langle\nabla\tilde{f}(x)\!-\!\nabla\tilde{f}(x^\star), x-x^\star\rangle\!\ge\! c\|x-x^\star\|^2, \forall x\in B(x^\star,\epsilon_1).
	\end{equation}
	Let $C\subset \mathbb{R}^d$ be any convex and compact set containing $x^\star$. Suppose $x\in C-B(x^\star,\epsilon_1)$. Let $z=x^\star+\epsilon_1\frac{x-x^\star}{\|x-x^\star\|}\in B(x^\star,\epsilon_1)$. Since $\tilde{f}$ is convex, $x-x^\star=\frac{x-z}{1-\epsilon_1/\|x-x^\star\|}$, and $\|x-x^\star\|>\epsilon_1$, we have $\langle \nabla \tilde{f}(x)-\nabla \tilde{f}(z), x-x^\star\rangle\ge 0$. Moreover, using $x-x^\star=\frac{\|x-x^\star\|}{\epsilon_1}(z-x^\star)$, \eqref{eq:strongconvexitytildef}, and $\|z-x^\star\|=\epsilon_1$, we obtain $\langle \nabla \tilde{f}(z)-\nabla \tilde{f}(x^\star), x-x^\star\rangle=\frac{\|x-x^\star\|}{\epsilon_1}\langle \nabla \tilde{f}(z)-\nabla \tilde{f}(x^\star), z-x^\star\rangle\ge\frac{c\|x-x^\star\|}{\epsilon_1}\|z-x^\star\|^2=\frac{c\epsilon_1}{\|x-x^\star\|}\|x-x^\star\|^2$. Adding the above two inequalities leads to $\langle \nabla \tilde{f}(x)-\nabla \tilde{f}(x^\star), x-x^\star\rangle\ge\frac{c\epsilon_1}{\max_{y\in C}\|y-x^\star\|}\|x-x^\star\|^2$. This, along with \eqref{eq:strongconvexitytildef}, implies that $\tilde{f}$ is restricted strongly convex with respect to $x^\star$ on $C$, so that Assumption~\ref{asm:stronglyconvex} holds.

\subsection{Proof of Lemma~\ref{lemma:levelset}}\label{ssec:proofoflevelset}

For simplicity, define $c^k:=\|\mathbf{z}^k-\mathbf{z}^\star\|_{\tilde{G}}^2+\rho\|\mathbf{x}^k\|_{\tilde{H}}^2$. Since $\tilde{H}\mathbf{x}^\star=\mathbf{0}$, $\|\mathbf{x}^k-\mathbf{x}^\star\|_{A_a+\rho\tilde{H}}^2=\|\mathbf{x}^k-\mathbf{x}^\star\|_{A_a}^2+\rho\|\mathbf{x}^k\|_{\tilde{H}}^2\le c^k$. Thus, to show $\mathbf{x}^k\in \mathcal{C}$ $\forall k\ge 0$, it suffices to show that $c^k\le c^0$ $\forall k\ge 0$. We prove this by induction. Clearly, $c^k\le c^0$ holds for $k=0$. Now suppose $c^k\le c^0$ and below we show $c^{k+1}\le c^k\le c^0$. Using \eqref{eq:vrelationforboundedz} and \eqref{eq:zkminuszstarexpansion} with $G$ replaced by $\tilde{G}$,
\begin{align}
&\|\mathbf{z}^k-\mathbf{z}^\star\|_{\tilde{G}}^2-\|\mathbf{z}^{k+1}-\mathbf{z}^\star\|_{\tilde{G}}^2-\|\mathbf{z}^{k+1}-\mathbf{z}^k\|_{\tilde{G}}^2\nonumber\displaybreak[0]\\
&\ge 2\eta\langle \mathbf{x}^k\!-\!\mathbf{x}^\star, \nabla f(\mathbf{x}^k)\!-\!\nabla f(\mathbf{x}^\star)\rangle\nonumber\displaybreak[0]\\
&\!\!\!+2\langle \mathbf{x}^{k+1}\!-\!\mathbf{x}^\star, (A^k\!-\!A_a)(\mathbf{x}^{k+1}\!-\!\mathbf{x}^k)\rangle\!-\!\tfrac{\|\mathbf{x}^{k+1}\!-\mathbf{x}^k\|_{\Lambda_M}^2}{2(1-\eta)}.\label{eq:rstrictedstrongztminuszstar}
\end{align}
On the other hand, since $A_\ell\preceq A^k\preceq A_u$, we have $\|A^k-A_a\|\le \Delta/2$. Due to \eqref{eq:restrictedstrongstepsize}, there exist $\beta>0$ and $\sigma\in(0,1)$ such that \eqref{eq:sigmabetarange} holds. Then, through the AM-GM inequality, $\langle \mathbf{x}^{k+1}-\mathbf{x}^\star, (A^k-A_a)(\mathbf{x}^{k+1}-\mathbf{x}^k)\rangle=\langle \mathbf{x}^k-\mathbf{x}^\star, (A^k-A_a)(\mathbf{x}^{k+1}-\mathbf{x}^k)\rangle+\langle \mathbf{x}^{k+1}-\mathbf{x}^k, (A^k-A_a)(\mathbf{x}^{k+1}-\mathbf{x}^k)\rangle\ge-\frac{1}{2}(\beta\Delta\|\mathbf{x}^k-\mathbf{x}^\star\|^2+\frac{\Delta}{4\beta}\|\mathbf{x}^{k+1}-\mathbf{x}^k\|^2)-\frac{\Delta}{2}\|\mathbf{x}^{k+1}\!-\!\mathbf{x}^k\|^2$, which further implies
\begin{align*}
&\|\mathbf{x}^{k+1}-\mathbf{x}^k\|_{A_a}^2\!+\!2\langle \mathbf{x}^{k+1}-\mathbf{x}^\star, (A^k-A_a)(\mathbf{x}^{k+1}-\mathbf{x}^k)\rangle\displaybreak[0]\\
&-\!\tfrac{\|\mathbf{x}^{k+1}-\mathbf{x}^k\|_{\Lambda_M}^2}{2(1-\eta)}\!\ge\!(1\!-\!\sigma)\|\mathbf{x}^{k+1}\!-\!\mathbf{x}^k\|_{A_a}^2\!-\!\beta\Delta\|\mathbf{x}^k\!-\!\mathbf{x}^\star\|^2. 
\end{align*}
Substituting this into \eqref{eq:rstrictedstrongztminuszstar} and applying \eqref{eq:AMMdual}, we obtain
\begin{align}
&c^k-c^{k+1}\ge (1-\sigma)\|\mathbf{x}^{k+1}-\mathbf{x}^k\|_{A_a}^2-\beta\Delta\|\mathbf{x}^k-\mathbf{x}^\star\|^2\nonumber\displaybreak[0]\\
&+2\eta\langle \mathbf{x}^k-\mathbf{x}^\star, \nabla f(\mathbf{x}^k)-\nabla f(\mathbf{x}^\star)\rangle+\rho\|\mathbf{x}^k\|_{\tilde{H}}^2.\label{eq:monotonekeyeq}
\end{align}
Due to the hypothesis $c^k\le c^0$, we have $\mathbf{x}^k\in \mathcal{C}$. It then follows from Proposition~\ref{prop:strongconvexity} and $\tilde{H}\mathbf{x}^\star=\mathbf{0}$ that $2\langle \mathbf{x}^k-\mathbf{x}^\star, \nabla f(\mathbf{x}^k)-\nabla f(\mathbf{x}^\star)\rangle \ge 2m_\rho\|\mathbf{x}^k-\mathbf{x}^\star\|^2-\rho\|\mathbf{x}^k\|_{\tilde{H}}^2$. This and \eqref{eq:monotonekeyeq} give
\begin{align}
&c^k-c^{k+1}\ge(1-\sigma)\|\mathbf{x}^{k+1}-\mathbf{x}^k\|_{A_a}^2\nonumber\displaybreak[0]\\
&+(2\eta m_\rho-\beta\Delta)\|\mathbf{x}^k-\mathbf{x}^\star\|^2+\rho(1-\eta)\|\mathbf{x}^k\|_{\tilde{H}}^2.\label{eq:strongzkbound}
\end{align}
Note that $\sigma\in(0,1)$, $\beta\Delta< 2\eta m_\rho$ due to \eqref{eq:sigmabetarange}, and $\eta\in (0,1)$. Thus, the right-hand side of \eqref{eq:strongzkbound} is nonnegative, which implies $c^{k+1}\le c^k\le c^0$.

\subsection{Proof of Theorem~\ref{theo:linearrate}}\label{ssec:proofoflinearrate}

Recall from the paragraph above Theorem~\ref{theo:linearrate} that we force $\mathbf{v}^0-\mathbf{v}^\star\in S^\bot$. Also, due to \eqref{eq:AMMdual}, $\mathbf{v}^k-\mathbf{v}^0\in \operatorname{Range}(\tilde{H}^{\frac{1}{2}})=S^\bot$. Hence, $\mathbf{v}^k-\mathbf{v}^\star\in \operatorname{Range}(\tilde{H}^{\frac{1}{2}})$. Moreover, since $h(\mathbf{x})\equiv0$, $g^{k+1}=g^\star=\mathbf{0}$. It then follows from \eqref{eq:squareroottildeHvstarminusvk} that $\mathbf{v}^\star-\mathbf{v}^k= (\tilde{H}^{\frac{1}{2}})^\dag(\nabla f(\mathbf{x}^k)-\nabla f(\mathbf{x}^\star)+(A^k+\rho H)(\mathbf{x}^{k+1}-\mathbf{x}^k)+\rho H\mathbf{x}^k)$. Then, via the AM-GM inequality, for any $\theta_1,\theta_2>0$,
\begin{align}
&\|\mathbf{v}^k-\mathbf{v}^\star\|^2 \le  (1+\theta_1)(1+\theta_2)\|\nabla f(\mathbf{x}^k)-\nabla f(\mathbf{x}^\star)\|_{\tilde{H}^\dag}^2\nonumber\displaybreak[0]\\
&\quad +2(1+\theta_1)(1+1/\theta_2)\|\mathbf{x}^{k+1}-\mathbf{x}^k\|_{(A^k)^T\tilde{H}^\dag A^k+\rho^2 H\tilde{H}^\dag H}^2\nonumber\displaybreak[0]\\
&\quad+\rho^2(1+1/\theta_1)\|H\mathbf{x}^k\|_{\tilde{H}^\dag}^2.\label{eq:tildeHsquarerootvkminusvstar}
\end{align}
From \eqref{eq:tildeHsquarerootvkminusvstar}, $H\tilde{H}^\dag H\preceq \bar{\lambda}H^{\frac{1}{2}}\tilde{H}^\dag H^{\frac{1}{2}}$, $\tilde{H}^\dag\preceq \frac{I_{Nd}}{\lambda_{\tilde{H}}}$, and the $M_i$-smoothness of each $f_i$, we have that for all $\delta\in(0,1)$,
\begin{align*}
&\delta\|\mathbf{z}^k-\mathbf{z}^\star\|_{\tilde{G}}^2\le\delta\|\mathbf{x}^k-\mathbf{x}^\star\|_{\frac{(1+\theta_1)(1+\theta_2)\Lambda_M^2}{\rho\lambda_{\tilde{H}}}+A_a}^2\nonumber\displaybreak[0]\\
&+\!\tfrac{2\delta(1\!+\!\theta_1)(1\!+\!1/\theta_2)}{\rho}\|\mathbf{x}^{k+1}\!-\!\mathbf{x}^k\|_{A_u^2/\lambda_{\tilde{H}}\!+\!\rho^2\bar{\lambda} H^{\frac{1}{2}}\tilde{H}^\dag H^{\frac{1}{2}}}^2\displaybreak[0]\\
&+\delta\rho(1+1/\theta_1)\|H\mathbf{x}^k\|_{\tilde{H}^\dag}^2.
\end{align*}
Note that \eqref{eq:strongzkbound} holds here because \eqref{eq:restrictedstrongstepsize} guarantees \eqref{eq:sigmabetarange} with some $\beta>0$ and $\sigma\in (0,1)$. Then, by combining the above inequality and \eqref{eq:strongzkbound}, we obtain $(1\!-\!\delta)(\|\mathbf{z}^k\!-\!\mathbf{z}^\star\|_{\tilde{G}}^2\!+\!\rho\|\mathbf{x}^k\|_{\tilde{H}}^2)\!-\!(\|\mathbf{z}^{k+1}\!-\!\mathbf{z}^\star\|_{\tilde{G}}^2\!+\!\rho\|\mathbf{x}^{k+1}\|_{\tilde{H}}^2)\ge\|\mathbf{x}^k-\mathbf{x}^\star\|_{B_1(\delta)}^2+\|\mathbf{x}^k\|_{B_2(\delta)}^2+\|\mathbf{x}^{k+1}-\mathbf{x}^k\|_{B_3(\delta)}^2$. Note from \eqref{eq:sigmabetarange} that $2\eta m_\rho-\beta\Delta>0$ and $A_a\succ\mathbf{O}$. Also note that $\theta_1>0$, $\theta_2>0$, $\rho>0$, $\lambda_{\tilde{H}}>0$, $\eta\in(0,1)$, $\tilde{H}\succeq \mathbf{O}$, $\operatorname{Range}(\tilde{H})=\operatorname{Range}(H)=S^\bot$, and $\sigma\in(0,1)$. Therefore, there exists $\delta\in (0, 1)$ such that $B_i(\delta)\succeq \mathbf{O}$ $\forall i\in \{1,2,3\}$, which guarantees \eqref{eq:theononcompositeresult}.

\subsection{Proof of Theorem~\ref{thm:sublinearstronglyconvex}}\label{ssec:proofofnonsmoothstronglyconvex}

In the proof of Lemma~\ref{lemma:levelset}, we have shown that $(c^k)_{k=0}^\infty$ is non-increasing. Thus, $\|\mathbf{z}^k-\mathbf{z}^\star\|_{\tilde{G}}^2\le c^k\le c^0=d_0^2/\rho$. It follows that $\|\mathbf{v}^0-\mathbf{v}^k\|\le \|\mathbf{v}^0-\mathbf{v}^\star\|+\|\mathbf{v}^\star-\mathbf{v}^k\|\le \|\mathbf{v}^0-\mathbf{v}^\star\|+\sqrt{\rho}\|\mathbf{z}^k-\mathbf{z}^\star\|_{\tilde{G}}\le \|\mathbf{v}^0-\mathbf{v}^\star\|+d_0$. By substituting this into \eqref{eq:consensuserrornonstronglyconvex} and \eqref{eq:general_funcvalge}, we obtain \eqref{eq:consensuserrornonsmoothstrongconvex} and \eqref{eq:objectiveerrorlowerboundnonsmoothstrongconvex}. Note that this substitution is legitimate, since \eqref{eq:consensuserrornonstronglyconvex} and \eqref{eq:general_funcvalge} do not rely on \eqref{eq:constantmatrix} to hold. To prove \eqref{eq:objectiveerrorupperboundnonsmoothstrongconvex}, we let $\tilde{J}^t(\mathbf{x})=u^t(\mathbf{x})+h(\mathbf{x})+\frac{\rho}{2}\|\mathbf{x}\|_H^2+(\mathbf{v}^t)^T\tilde{H}^{\frac{1}{2}}\mathbf{x}$ $\forall t\ge0$. Since $\nabla^2 u^t(\mathbf{x})\succeq A_\ell$ $\forall \mathbf{x}\in \mathbb{R}^{Nd}$, $\tilde{J}^t(\mathbf{x})-\|\mathbf{x}\|_{A_\ell}^2/2$ is convex. Similar to the derivation of \eqref{eq:constantmatrixxminimize},
\begin{equation}\label{eq:stronglyconvexJtJstar}
\tilde{J}^t(\mathbf{x}^{t+1})-\tilde{J}^t(\mathbf{x}^\star)\le -\|\mathbf{x}^{t+1}-\mathbf{x}^\star\|_{A_\ell}^2/2.
\end{equation}
Since $A_\ell\preceq \nabla^2 u^t(\mathbf{x})\preceq A_u$ $\forall \mathbf{x}\in \mathbb{R}^{Nd}$, we have $u^t(\mathbf{x}^{t+1})\ge u^t(\mathbf{x}^t)+\langle \nabla u^t(\mathbf{x}^t), \mathbf{x}^{t+1}-\mathbf{x}^t\rangle+\|\mathbf{x}^{t+1}-\mathbf{x}^t\|_{A_\ell}^2/2$ and $u^t(\mathbf{x}^\star) \le u^t(\mathbf{x}^t)+\langle \nabla u^t(\mathbf{x}^t), \mathbf{x}^\star-\mathbf{x}^t\rangle+\|\mathbf{x}^\star-\mathbf{x}^t\|_{A_u}^2/2$. These two inequalities along with $\nabla u^t(\mathbf{x}^t)=\nabla f(\mathbf{x}^t)$ imply
\begin{align}
&u^t(\mathbf{x}^{t+1})-u^t(\mathbf{x}^\star)\ge \langle \nabla f(\mathbf{x}^t), \mathbf{x}^{t+1}-\mathbf{x}^\star\rangle\nonumber\displaybreak[0]\\
&+\|\mathbf{x}^{t+1}-\mathbf{x}^t\|_{A_\ell}^2/2-\|\mathbf{x}^\star-\mathbf{x}^t\|_{A_u}^2/2.\label{eq:ut1minusustar}
\end{align}
Note that \eqref{eq:xk1tildeHxvk} and \eqref{eq:fxk1minusfstarupperbound} hold in no need of \eqref{eq:constantmatrix}. Hence, by integrating \eqref{eq:stronglyconvexJtJstar}, \eqref{eq:ut1minusustar}, \eqref{eq:xk1tildeHxvk}, \eqref{eq:fxk1minusfstarupperbound}, and $H\mathbf{x}^\star\!=\!\tilde{H}^{\frac{1}{2}}\mathbf{x}^\star\!=\!\mathbf{0}$, $f(\mathbf{x}^{t+1})\!+\!h(\mathbf{x}^{t+1})\!-\!f(\mathbf{x}^\star)\!-\!h(\mathbf{x}^\star)\!+\!\frac{\|\mathbf{v}^{t+1}\|^2-\|\mathbf{v}^t\|^2}{2\rho}\le\frac{\|\mathbf{x}^t-\mathbf{x}^\star\|_{A_u}^2}{2}-\frac{\|\mathbf{x}^{t+1}-\mathbf{x}^\star\|_{A_\ell}^2}{2}+\frac{\|\mathbf{x}^{t+1}-\mathbf{x}^t\|_{\Lambda_M-A_\ell}^2}{2}$. Similar to the derivation of \eqref{eq:general_funcvalle} from \eqref{eq:nonstrongfplushoptimaldistance}, it follows that $f(\bar{\mathbf{x}}^k)+h(\bar{\mathbf{x}}^k)-f(\mathbf{x}^\star)-h(\mathbf{x}^\star)\le\frac{1}{k}\bigl(\frac{\|\mathbf{v}^0\|^2}{2\rho}+\frac{\|\mathbf{x}^0-\mathbf{x}^\star\|_{A_u}^2}{2}\bigr)+\frac{1}{k}\sum_{t=0}^{k-1}\frac{\|\mathbf{x}^{t+1}-\mathbf{x}^t\|_{\Lambda_M-A_{\ell}}^2}{2}+\frac{\Delta}{2k}\sum_{t=1}^{k-1}\|\mathbf{x}^t-\mathbf{x}^\star\|^2$. From \eqref{eq:strongzkbound}, $\sum_{t=0}^{k-1}\|\mathbf{x}^{t+1}\!-\!\mathbf{x}^t\|_{A_a}^2\le\frac{c^0-c^k}{1-\sigma}\le\frac{d_0^2}{\rho(1-\sigma)}$ and $\sum_{t=1}^{k-1}\|\mathbf{x}^t-\mathbf{x}^\star\|^2\le\frac{c^1-c^k}{2\eta m_\rho-\beta\Delta}\le\frac{d_0^2}{\rho(2\eta m_\rho-\beta\Delta)}$, where $\beta>0$ and $\sigma\in(0,1)$ satisfying \eqref{eq:sigmabetarange} exist due to \eqref{eq:restrictedstrongstepsize}. Also, from \eqref{eq:restrictedstrongstepsize}, $\Lambda_M-A_\ell\preceq \Lambda_M\preceq 2(1-\eta)A_a$. Combining the above gives \eqref{eq:objectiveerrorupperboundnonsmoothstrongconvex}.

\bibliography{reference}

\bibliographystyle{ieeetran}

\begin{IEEEbiography}
[{\includegraphics[width=1in,height=1.25in,clip,keepaspectratio]{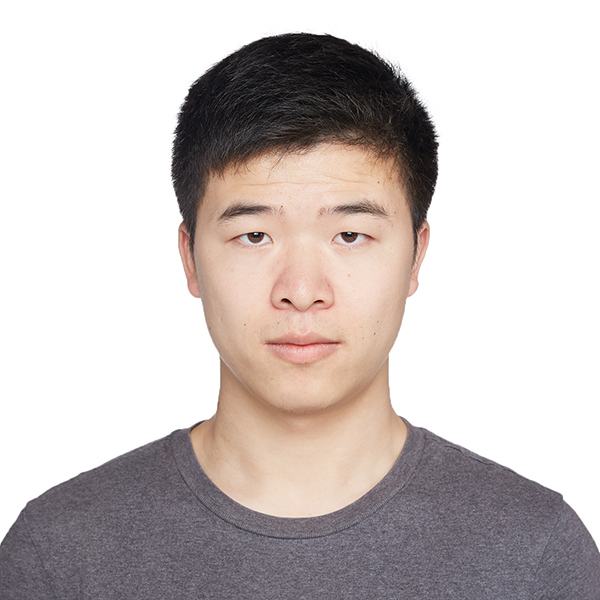}}]{Xuyang Wu}
(SM'17) received the B.S. degree in Information and Computing Science from Northwestern Polytechnical University, Xi'an, China, in 2015, and the Ph.D. degree from the School of Information Science and Technology at ShanghaiTech University, Shanghai, China, in 2020. He is currently a postdoctoral researcher in the Division of Decision and Control Systems at KTH Royal Institute of Technology, Stockholm, Sweden. His research interests include distributed optimization, large-scale optimization, and their applications in IoT and power systems.
\end{IEEEbiography}
	
\begin{IEEEbiography}[{\includegraphics[width=1in,height=1.25in,clip,keepaspectratio]{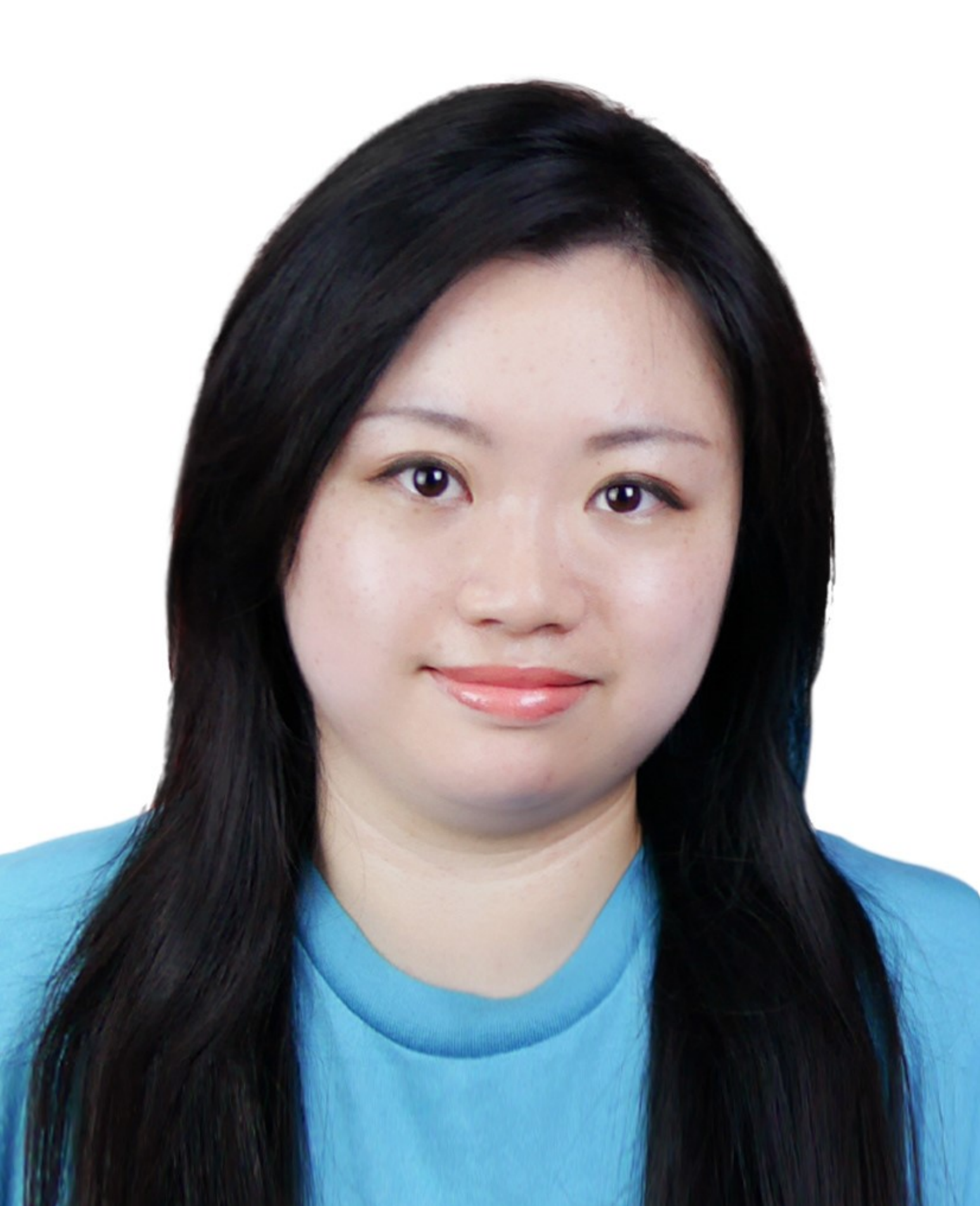}}]{Jie Lu} (SM'08-M'13)
received the B.S. degree in Information Engineering from Shanghai Jiao Tong University, China, in 2007, and the Ph.D. degree in Electrical and Computer Engineering from the University of Oklahoma, USA, in 2011. From 2012 to 2015 she was a postdoctoral researcher with KTH Royal Institute of Technology, Stockholm, Sweden, and with Chalmers University of Technology, Gothenburg, Sweden. Since 2015, she has been an assistant professor in the School of Information Science and Technology at ShanghaiTech University, Shanghai, China. Her research interests include distributed optimization, optimization theory and algorithms, and networked dynamical systems.
\end{IEEEbiography}

\end{document}